\documentclass{elsarticle}[12pt]

\usepackage{amsmath}
\usepackage{amsfonts}
\usepackage{amssymb}
\usepackage{amsthm}
\usepackage{amscd}
\usepackage{enumerate}
\usepackage{hyperref}
\usepackage{appendix}
\usepackage{comment}

\newtheorem{thm}{Theorem}[section]
\newtheorem{lem}[thm]{Lemma}
\newtheorem{cor}[thm]{Corollary}
\newtheorem{prop}[thm]{Proposition}

\usepackage{todonotes}
\usepackage{xcolor}
\usepackage[normalem]{ulem}
\usepackage{cancel}

\theoremstyle{definition}

\newtheorem{definition}[thm]{Definition}
\theoremstyle{remark}

\numberwithin{equation}{section}

\newcommand{\bb}{\mathbb}

\newcommand{\mrm}{\mathrm}
\newcommand{\mfk}{\mathfrak}
\newcommand{\mcal}{\mathcal}
\newcommand{\gnk}{\Gamma_{n,k}}

\journal{Discrete Mathematics}

\begin{document}

\begin{frontmatter}

\title{Hilton-Milner Theorem for the $r$-independent sets in a union of cliques}


\author{Karen Gunderson\fnref{KGNSERC}} 
\ead{Karen.Gunderson@umanitoba.ca}

\affiliation{organization={Department of Mathematics},
            addressline={University of Manitoba}, 
            city={Winnipeg},
            postcode={R3T-2N2}, 
            state={Manitoba},
            country={Canada}}
\fntext[KGNSERC]{Research supported in part by an NSERC Discovery Research Grant,   Application No.: RGPIN-2024-06029.}
    
\author{Karen Meagher\fnref{KMNSERC}}
\ead{karen.meagher@uregina.ca}
\affiliation{organization={Department of Mathematics and Statistics},
            addressline={University of Regina}, 
            city={Regina},
            postcode={S4S-0A2}, 
            state={Saskatchewan},
            country={Canada}}
\fntext[KMNSERC]{Research supported in part by an NSERC Discovery Research Grant,   Application No.: RGPIN-2025-04101.}
    
\author{Joy Morris\fnref{JMNSERC}}
\ead{joy.morris@uleth.ca}
\affiliation{organization={Department of Mathematics and Computer Science},
            addressline={University of Lethbridge}, 
            city={Lethbridge},
            postcode={T1K-3M4}, 
            state={Alberta},
            country={Canada}}

\fntext[JMNSERC]{Research supported in part by an NSERC Discovery Research Grant,   Application No.: RGPIN-2024-04013.}

\author{Venkata Raghu Tej Pantangi}
\ead{pvrt1990@gmail.com}



\begin{abstract}  
We give a Hilton-Milner Theorem for the $r$-independent sets in the graph that is the union of copies of $K_k$. That is, we determine the maximum intersecting families of $r$-independent sets in this graph, subject to the condition that the sets in a family do not all share a common element. As a by-product, we also find a tight upper bound for the sum of sizes of a pair of cross intersecting families made up of the same objects.  

We apply our theorem to find the largest intersecting family of $r$-independent sets in a family of graphs called ``depth-two claws". This confirms the Holroyd--Talbot conjecture for depth-two claws, extending previous results on these graphs (which covered cases where $r$ was relatively small compared to the number of vertices) to all possible values of $r$.
\end{abstract}

\begin{keyword}
Erd\H{o}s-Ko-Rado Theorem \sep Hilton-Milner Theorem \sep Intersecting Families of Independent Sets \sep Cross-Intersecting Families \sep Holroyd-Talbot conjecture \sep Depth-2 Claws

 \MSC  05C35 \sep 05C69 \sep 05D05 
\end{keyword}

\end{frontmatter}


\section{Introduction}
\label{sec:Intro}

An $r$-independent set in a graph is an independent set of size $r$. A family of $r$-independent sets is said to be \textsl{intersecting} if every pair of sets in the family has non-empty intersection. A family consisting of all $r$-independent sets in a graph that contain a fixed vertex is said to be \textsl{canonically intersecting}. A graph $G$ is said to have the \textsl{$r$-EKR property}, or equivalently to be \textsl{$r$-EKR}, if the maximum size for a family of intersecting $r$-independent sets is the size of a canonically intersecting set.  A graph is said to have the \textsl{strict $r$-EKR property} if the \emph{only} maximum-size intersecting families of $r$-independent sets are canonically intersecting. These names derive from the well-known Erd\H{o}s-Ko-Rado Theorem which is equivalent to the statement that the empty graph on $n$ vertices has the $r$-EKR property for $2r \leq n$, and the strict EKR property if $2r < n$.

There is a 20 year-old conjecture that predicts that many graphs have the EKR property. Letting $\mu(G)$ denote the size of a smallest maximal independent set in a graph $G$, Holroyd and Talbot~\cite{HolroydTalbot} conjectured that for any graph $G$, if $r < \mu(G)/2$, then $G$ is $r$-EKR.

Define $\gnk = \cup_{i=1}^n K_k$ to be the graph that is the disjoint union of $n$ copies of a complete graph on $k$ vertices.  In $\gnk$, any $r$-independent set  is comprised of one vertex taken from each of exactly $r$ of the copies of $K_k$. Denote the collection of all $r$-independent sets in $\gnk$ by $\mcal{I}_{n,k}^r$. Then $| \mcal{I}_{n,k}^r | =\binom{n}{r}  k^r$.

In this paper, we consider maximal families of intersecting $r$-independent sets in $\gnk$.  In the case $r = n$, Berge~\cite{cB74} and Livingston~\cite{mL79} showed that $\gnk$ is $r$-EKR, and strictly $r$-EKR, except in the case $k = 2$.
Meyer~\cite{jM74} stated a generalization to any $r \leq n$, which was proved by Deza and Frankl~\cite{DF22years}: $\gnk$ is $r$-EKR. This means, if $\mathcal{F} \subseteq \mcal{I}_{n,k}^r$ is intersecting, then 
\[
| \mathcal{F} | \leq \binom{n-1}{r-1} k^{r-1},
\]
since this is the number of $r$-independent sets that contain a given vertex.

Further, Deza and Frankl~\cite{DF22years} proved that, except when $r= n$ and $k = 2$, $\gnk$ is strictly $r$-EKR.  Holroyd, Spencer, and Talbot~\cite{HST05} further generalized this result to a setting where the cliques are of different sizes.

In this paper we consider the natural version of the Hilton-Milner Theorem for the $r$-independent sets in $\gnk$: determining the size of the largest collection $\mathcal{F}  \subseteq \mcal{I}_{n,k}^r$ that is intersecting, and additionally has $\cap \mcal{F}  = \emptyset$ (that is, the intersection of all the sets in $\mcal{F}$ is empty, so the sets in $\mcal{F} $ do not all contain a common element).
We then show that this result can be used extend the range of $r$ for which a related graph known as a ``depth-two claw" is known to be $r$-EKR, which was studied by Feghali, Johnson, and Thomas~\cite{claws}.

We now present history about the original Erd\H{o}s-Ko-Rado and Hilton-Milner Theorems, and related results; after this background we will outline the structure of the rest of the paper. 

\section{Background}

The Erd\H{o}s-Ko-Rado (EKR) Theorem  is a cornerstone result in combinatorics and extremal set theory. It was originally proven in 1961~\cite{EKR1961}, and over the years there have been many generalizations, extensions, and applications. Following standard notation, we use $[n]$ to denote $\{1, \ldots, n\}$ and for any integers $a \leq b$, denote $[a, b] = \{i \in \bb{Z}\ :\ a\leq i \leq b\}$.  For any set $A$ and $r \geq 0$, let $\binom{A}{r}$ denote the collection of all $r$-sets from $A$, and $\binom{A}{\leq r}$ the collection of subsets of size no more than $r$.
We say a family $\mathcal{B} \subseteq \binom{[n]}{r}$ is \textsl{intersecting} if $B_1 \cap B_2 \neq \emptyset$ for all $B_1, B_2 \in \mathcal{B}$. The collection of all $r$-sets that contain a fixed element is clearly an intersecting set of size $\binom{n-1}{r-1}$---this is called \textsl{a canonically intersecting set}.
We denote the intersection of all sets in $\mathcal{B}$ by $\cap \mathcal{B} = \cap_{B \in \mathcal{B}} B$; if $\cap \mathcal{B} \neq \emptyset$, then $\mathcal{B}$ is a subset of a canonically intersecting set.

Two collections of $r$-sets are \textsl{isomorphic} if one family can be produced from the other by a relabelling of the elements in $[n]$, so all canonically intersecting sets are isomorphic. If families $\mathcal{F}_1$ and $\mathcal{F}_2$ are isomorphic, then we write $\mathcal{F}_1 \cong \mathcal{F}_2$.

\begin{thm}[Erd\H{o}s-Ko-Rado (EKR) Theorem~\cite{EKR1961}]\label{thm:EKRsets}
Let $n \geq 2r$, and $\mathcal{B} \subseteq \binom{[n]}{r}$. If $\mathcal{B}$ is intersecting, then $|\mathcal{B}| \leq \binom{n-1}{r-1}$. Moreover, if $n>2r$, then $|\mathcal{B}| = \binom{n-1}{r-1}$ if and only if $\mathcal{B}$ is a canonically intersecting set.
\end{thm}

One of the earliest generalizations of the EKR Theorem was by Hilton and Milner, who determined the size and structure of the largest intersecting family from $\binom{[n]}{r}$ with the additional condition that the sets in the family do not all contain a common element. Define 
\[
\mathcal{H}_{n,r} = \left \{ B \in \binom{[n]}{r} \, : \, 1 \in B , \ B \cap [2, r+1] \neq \emptyset  \right\} \cup 
\left\{ [2, r+1] \right\}.
\]
This family is formed by taking a canonical intersecting family, adding an additional $r$-set, and removing the sets from the canonical family that don't intersect with the additional set. The sets being removed are the sets that contain $1$ and all of whose other elements are drawn from $[r+2, n]$, so there are $\binom{n-r-1}{r-1}$ of these sets. We call this set the \textsl{HM set}.

\begin{thm}[Hilton-Milner Theorem~\cite{HM}]
\label{thm:HMforSets}
Let $n \geq  2r \geq 4$ and $\mathcal{B} \subseteq \binom{[n]}{r}$. If $\mathcal{B}$ is intersecting and 
$\cap \mathcal{B}  = \emptyset$, then
\[
|\mathcal{B}| \leq \binom{n-1}{r-1} - \binom{n-r-1}{r-1} + 1.
\]
Moreover, if $n >  2r > 6$, then
$|\mathcal{B}| = \binom{n-1}{r-1} - \binom{n-r-1}{r-1} + 1$ if and only if $\mathcal{B}$ is isomorphic to $\mathcal{H}_{n,r}$.
\end{thm}

This theorem shows that there is a significant ``jump" down in size from a canonically intersecting family, to the largest intersecting family that is not contained in a canonically intersecting family; it is this gap that we will be exploiting in our application.

We will also need to consider questions related to ``cross-intersecting families," which we now define. A pair $(\mcal{A}, \mcal{B})$ of non-empty families of non-empty subsets of $[n]$ is said to be \textsl{cross-intersecting} if for all $A \in \mcal{A}$ and $B \in \mcal{B}$ it follows that $A \cap B \neq \emptyset$. We say two cross-intersecting families are \textsl{isomorphic} if one family can be obtained from the other, via a relabelling operation on $[n]$. The following analogue of Theorem~\ref{thm:EKRsets} (the EKR Theorem) was obtained by Hilton and Milner~\cite{HM} and rediscovered by Simpson in~\cite{SIMPSON1993}. It shows that cross-intersecting families are typically maximised by taking one of $\mcal{A}$ or $\mcal{B}$ to consist of a single set, and the other to consist of all sets that intersect nontrivially with that single set.

\begin{thm}[Hilton--Milner~\cite{HM}, Simpson~\cite{SIMPSON1993}]\label{thm:hmci}
Let $r \geq 1$ and $n\geq 2r$.
If $(\mcal{A}, \mcal{B})$ is a cross-intersecting pair in $\binom{[n]}{r}$ with $\mcal{A}, \mcal{B} \neq \emptyset$, then
\[
|\mcal{A}| + |\mcal{B}| \leq 1+  \binom{n}{r}- \binom{n-r}{r}.
\] 
For $n > 2r$, equality holds if and only if  $(\mcal{A}, \mcal{B})$ is isomorphic to 
\[
\left(\{[r]\},\ \{X \in \binom{[n]}{r}\ :\ X \cap [r] \neq \emptyset\} \right).
\] 
\end{thm}

Using the Kruskal--Katona theorem, Frankl and Tokushige~\cite{frankl1992some} generalized Theorem~\ref{thm:hmci} to the case where $\mcal{A} \subseteq \binom{[n]}{r}$ and $\mcal{B} \subseteq \binom{[n]}{s}$ (with $r$ and $s$ not necessarily equal) and a recent related result of Borg and Feghali~\cite{borg2022maximum} solves the problem for $\mcal{A} \subseteq \binom{[n]}{\leq r}$ and $\mcal{B} \subseteq \binom{[n]}{\leq s}$.

Theorems about cross-intersecting pairs of families have been used to give a variety of proofs of Theorem~\ref{thm:HMforSets} (see, e.g., \cite{bulavka2025, pF19, frankl1992some, HK18, wu2024stabilities}), so it is not surprising that our technique for proving our main theorem also takes this approach.

In Section~\ref{sect:main} we will present the statements of our main results: Theorem~\ref{thm:main} is our adaptation of the Hilton-Milner Theorem to $\gnk$, while Theorem~\ref{thm:aux} is a new version of the Hilton-Milner theorem on cross-intersecting sets in the context of $\gnk$. We require the latter for our proof of the former, but believe that Theorem~\ref{thm:aux} is also of interest in its own right. The other result in this section (Lemma~\ref{r=3}) shows that the bound $r \ge 4$ in the statement of Theorem~\ref{thm:main} is best possible. Section~\ref{sec:compandpro} gives a thorough explanation of the operations of compression and projection that we will be using in our proofs. Many previous EKR- or HM-type results have been been vague about the details of their compression operations and this has led to errors.
Accordingly, we take care to lay out our operations and prove results about them with rigour. 

Section~\ref{sec:crsets} proves results about intersecting and cross-intersecting families in the generalised situation where the sets have size at most $r$, rather than necessarily having size $r$. Some of these results are careful re-proofs of results that appear in other papers, but in which details were omitted or the original statements do not make it clear that they apply to our circumstances. Along similar lines, Section~\ref{sect:ApplyHM} is devoted to a full proof of Lemma~\ref{lem:ApplyHM}, although this lemma is essentially a restatement of results that have appeared elsewhere. These results will be required since the compression operation we employ can reduce the size of a set.  Section~\ref{sec:crintsets} contains the proof of Theorem~\ref{thm:aux}, while Section~\ref{sect:mainproof} proves Theorem~\ref{thm:main}. In Section~\ref{sect:app} we apply our main theorem to the situation of depth-two claws. Section~\ref{sec:Further} concludes the paper by outlining possible future work.


\section{Main results}\label{sect:main}

Let $\gnk = K_k^{(1)} \cup K_k^{(2)} \cup \dots \cup K_k^{(n)}$ be a union of $n$ disjoint copies of $K_k$ and label the vertices in $\gnk$ by a tuple $(i,j)$ where $i \in [n]$ represents the copy of $K_k$ that contains the vertex, and $j \in [k]$ represents the specific vertex in $K_k^{(i)}$. Recall that $ \mcal{I}_{n,k}^r $ is the collection of all $r$-independent sets in $\gnk$, and define $\mathcal{E}_{n,k}^r = \{ F \in  \mcal{I}_{n,k}^r \, : \, (1,1) \in F \}$, so 
$|\mathcal{E}_{n,k}^r| = \binom{n-1}{ r -1} k^{r-1}$.
This is called the \textsl{Erd\H{o}s-Ko-Rado family} and Deza and Frankl~\cite[Theorem 5.2]{DF22years} showed that this is a largest family of intersecting $r$-independent sets in $\gnk$. Furthermore, they showed that, except in the case $k = 2$ and $r = n$, any other maximum family of intersecting $r$-independent sets is isomorphic to the Erd\H{o}s-Ko-Rado family.

We consider the largest intersecting families of $r$-independent sets that do not all contain a fixed vertex.  In the case $r \leq n-1$, define $$H = \{ (2,1), (3,1), \dots, (r+1,1) \} = [2, r+1] \times \{1\} \in \mcal{I}_{n,k}^r;$$ this is an $r$-independent set that is not in $\mathcal{E}_{n,k}^r$. Define the \textsl{Hilton--Milner family} (HM family) to be $H$ together with all sets in $\mcal{E}_{n,k}^r$ that intersect $H$:
\begin{equation}\label{eq:HM-family-in-clique-union}
\mathcal{H}_{n,k}^r = \{ F \in \mathcal{E}_{n,k}^r  \, : \, F \cap H \neq \emptyset \} \cup \{H\}.
\end{equation}
The size of $\mathcal{H}_{n,k}^r$ can be determined by counting the number of sets that must be removed from $\mathcal{E}_{n,k}^r$ because they do not intersect with $H$. In these sets, any vertex of the form $(a,b)$ with $a \in [2, r+1]$ must have $b \neq 1$. 
We can count these sets by counting the sets with exactly $j$ vertices of the form $(a,b)$ with $a \in [2, r+1]$ and $b \neq 1$: there are $\binom{r}{j}$ ways to choose the values for $a$, $(k-1)^j$ ways to choose the values for $b$, and then the remaining elements of the set can be chosen in $\binom{n-r-1}{r-j-1}k^{r-j-1}$ ways. So:
\[
| \mathcal{H}_{n,k}^r | =  \binom{n-1}{ r -1} k^{r-1} - \sum_{j=0}^{r-1} \binom{r}{j} (k-1)^j  \binom{n-r-1}{r-j-1} k^{r-j-1} +1.\]
Applying the Vandermonde binomial identity gives 
\[
\binom{n-1}{r-1}k^{r-1}=\left(\sum_{j=0}^{r-1}\binom{r}{j}\binom{n-r-1}{r-j-1}\right)k^{r-1}=\sum_{j=0}^{r-1}\binom{r}{j}\binom{n-r-1}{r-j-1}k^j k^{r-j-1},\]
and so 
\begin{align}
	| \mathcal{H}_{n,k}^r |=1 + \sum_{j = 1}^{r-1} \binom{r}{j}\binom{n-r-1}{r-j-1}k^{r-j-1}(k^j - (k-1)^j). \label{eq:card-HM-family}
\end{align}
(The sum could be written to start at $0$, but when $j=0$ we have $k^j-(k-1)^j=0$ so this term contributes nothing.)

The following is our main result. 
\begin{thm}\label{thm:main}
Let $3 \leq r \leq n$ and let $\mathcal{F} \subset \mcal{I}_{n,k}^r$ be an intersecting family with $\cap \mathcal{F} = \emptyset$. Then the following are true.
\begin{enumerate}
\item If $r<n$, then $|\mathcal{F}| \leq |\mathcal{H}_{n,k}^r |$. Moreover, if $r\geq 4$, equality holds if and only if $\mathcal{F}$ is isomorphic to $\mathcal{H}_{n,k}^r$.
\item If $r=n$, then $|\mathcal{F}| \leq k^{n-1} - (k-1)^{n-1} + k-1$. Moreover, if $n \geq 5$, equality holds if and only if $\mathcal{F}$ is isomorphic to 
\begin{multline*}
\{ X \in \mcal{I}_{n, k}^{n} :\ (1, 1) \in X \text{ and }  X \cap \left ( [2, n] \times \{1\} \right) \neq \emptyset\}\\
 \cup \big\{ \{(1, i)\} \cup ([2, n] \times \{1\}):\ i \geq 2  \big\}.
\end{multline*}
\end{enumerate}   
\end{thm}

Note that in the case $r = n$ and $k =2$, the upper bound gives $k^{n-1} - (k-1)^{n-1} + 1 = 2^{n-1}$, which is the same as the size of the Erd\H{o}s-Ko-Rado family $\mathcal{E}_{n, 2}^n$. The non-uniqueness of maximum intersecting families in this case was noted by Deza and Frankl~\cite{DF22years}.

Note that for $r = 2, n > 2$, the family $\mathcal{H}_{n, k}^2$ has cardinality $3$ and is the unique largest family, up to isomorphism.

\begin{lem}\label{r=3}
If $r=3$ there is an intersecting family with $\cap \mcal{F}  = \emptyset$,
and $|\mathcal{F}| = |\mathcal{H}_{n,k}^r |$ that is not isomorphic to $\mathcal{H}_{n,k}^r$.
\end{lem}
\begin{proof}
Let $\mathcal{F} \subset \mcal{I}_{n,k}^3$ be the family of all independent sets in $\mcal{I}_{n,k}^3$ that contain at least two of the vertices 
$\{ (1,1), (2,1), (3,1)\}$. This set is clearly intersecting, and as has size
$$
1 + 3 \left( (k-1) + (n-3)k \right) = 3kn-6k-2,
$$
which is the same size as $\mathcal{H}_{n,k}^3$. 

We now show that $\mcal{F}$ is not isomorphic to $\mcal{H}_{n,k}^{3}$. To do so, we note that $(1,1)$ is an element in $|\mcal{H}_{n,k}^{3}|-1=3 \left( (k-1) + (n-3)k \right)$ of the sets from $\mcal{H}_{n,k}^{3}$. To show $\mcal{F}\not\cong \mcal{H}^{3}_{n,k}$, it suffices to show that no single vertex is contained in $3 \left( (k-1) + (n-3)k \right)$ sets of $\mcal{F}$. We conclude the proof by observing that (i) any vertex in $\{(1,1),\ (2,1),\ (3,1)\}$ occurs in exactly $1+ 2((k-1)+(n-3)k)$ sets of $\mcal{F}$; (ii) a vertex of the form $(x,a)$ with $x \in [3,n]$ and $a \in [k]$, occurs in exactly $3$ sets of $\mcal{F}$; and (iii) a vertex of the form $(y,b)$ with $y \in [3]$ and $b \in [2,k]$, occurs in exactly in $1$ set of $\mcal{F}$. 
\end{proof}

An important tool in the proof of Theorem~\ref{thm:main} is an analogue of Theorem~\ref{thm:hmci} (the Hilton-Milner Theorem on cross-intersecting families) for cross-intersecting families of $r$-independent sets in $\gnk$.  A pair $(\mcal{A},\ \mcal{B})$ of non-empty families of independent sets in $\Gamma_{n,k}$ is said to be \textsl{cross-intersecting} if $A\cap  B \neq \emptyset$, for all $A \in \mcal{A}$ and $B \in \mcal{B}$. For any $r \leq n$, one example of one such pair is 
\begin{equation}\label{eq:HM-cross-intersecting}
\left( 
    \mfk{H}_{n,k}^{r}:=\left\{ [r]\times \{1\} \right\},\ 
    \mfk{M}_{n,k}^{r}:= \left\{  X \in \mcal{I}_{n,k}^{r}\ :\ X\cap ([r] \times \{1\})\neq \emptyset  \right\} 
    \right). 
\end{equation}
We call this the \textsl{Hilton--Milner cross-intersecting pair}. We show the following result on  cross-intersecting pairs in $\mcal{I}^{r}_{n,k}$.

\begin{thm}\label{thm:aux}
Let $n,k,r$ be positive integers with $2\leq r\leq n$. If $(\mcal{A},\ \mcal{B})$ is a cross-intersecting pair of families of $r$-independent sets in $\Gamma_{n,k}$, then  
\[
|\mcal{A}|+ |\mcal{B}| \leq |\mfk{H}_{n,k}^{r}| + |\mfk{M}_{n,k}^{r}|.
\] 
Moreover, when $r\geq 3$, equality holds if and only if $(\mcal{A},\ \mcal{B})\cong (\mfk{H}_{n, k}^r,\ \mfk{M}_{n, k}^r)$.

In the case where $r = 2$, equality holds if and only if either
\begin{enumerate}[(a)]
	\item $(\mcal{A}, \mcal{B}) \cong (\mfk{H}_{n, k}^2, \mfk{M}_{n, k}^2)$ or
	\item $\mcal{A} = \mcal{B} \cong \{X \in \mcal{I}_{n, k}^2:\ (1, 1) \in X\}$.
\end{enumerate}
\end{thm}

Proofs of Theorem~\ref{thm:main} and \ref{thm:aux} are inspired by the proofs in~\cite{DF22years} and~\cite{liao2024hiltonmilnertheoremkmultisets}. A key tool in our proofs is \textsl{set compression}, also known as \textsl{shifting}. We describe this in the following section. We will also need some results about cross-intersecting pairs and intersecting families in $\binom{[n]}{\leq r}$, described  in Section~\ref{sec:crsets}.

\section{Compression and projection.}\label{sec:compandpro}

In this section, we give a number of results on the effects of two operations: \textsl{compression} and \textsl{projection} on families of sets in $\mcal{I}_{n, k}^r$.  

Note that a family $\mathcal{A} \subseteq \mcal{I}_{n, k}^r$ is intersecting if and only if  the pair $(\mcal{A}, \mcal{A})$ is cross-intersecting.  We will generally prove results only for cross-intersecting families and then use the fact that a corresponding result for intersecting families follows immediately.

The original proof of the EKR theorem, in~\cite{EKR1961}, relied on set compression. Different versions of the compression operation have proved useful in other analogues of the EKR theorem and we refer the reader to the survey by Frankl~\cite{frankl1987shifting} on this method.

In the paper by Deza and Frankl~\cite{DF22years}, a compression operation called ``pushing-up'' is used on the independent sets in $\Gamma_{n,k}$, the goal of this operation is to try to replace any element of the $i$-th clique with the element $(i, 1)$.  We use a more granular version of their compression operation.  Given $i \in [n]$, $s \in [2,k]$ and $X \in \mcal{I}^{r}_{n,k}$, we define 
\[
P_{i,s}(X)= 
\begin{cases} 
X \setminus \{(i,s)\} \cup \{(i,1)\} & \text{if $(i,s) \in X$,} \\
X & \text{otherwise.}
\end{cases}
\]

Given a family $\mcal{F}$ of $r$-independent sets, we define 
$$
\pi_{i,s}(\mcal{F})= 
\left\{ P_{i,s}(X)\ :\ X \in \mcal{F} \right\} 
\bigcup 
\left\{X \ :\ X,\ P_{i,s}(X) \in \mcal{F} \right\}.
$$
That is, we replace each set in $\mcal{F}$ with its compressed version, unless that compressed version is already in $\mcal{F}$, in which case we keep both sets.

For completeness, we include a proof that the operation $\pi_{i, s}$ preserves cardinalities of families and intersection properties.

\begin{lem}\label{lem:pipreservescint}
For every $i \in [n]$, $s \in [2,k]$, and $\mcal{A}, \mcal{B} \subseteq \mcal{I}_{n, k}^r$:
\begin{enumerate}[(a)]
\item $\pi_{i, s}(\mcal{A}) \subseteq \mcal{I}_{n, k}^r$ with  $|\pi_{i,s}(\mcal{A})|= |\mcal{A}|$.
\item If $(\mcal{A},\ \mcal{B})$ is a cross-intersecting pair, then so is $\left(\pi_{i,s}(\mcal{A}),\ \pi_{i,s}(\mcal{B}) \right)$.
\end{enumerate}
\end{lem}
\begin{proof}
Part (a) follows immediately from the definition.

To prove part (b), assume $(\mcal{A},\ \mcal{B})$ are a cross-intersecting pair and let $A \in \mcal{A}$, $B \in \mcal{B}$ be arbitrary. We first show that $P_{i,s}(A)\cap P_{i,s}(B)\neq \emptyset$. Depending on whether or not $(i,s)\in A\cap B$, we either have $(i, 1) \in P_{i, s}(A) \cap P_{i, s}(B)$, or $A \cap B \subseteq P_{i,s}(A) \cap P_{i,s}(B)$. In either case we are done. 

We also need to show that if $A\neq P_{i,s}(A) \in \mcal{A}$, then $A \cap P_{i,s}(B) \neq \emptyset$. Since $(\mcal A, \mcal B)$ is cross-intersecting, we have $A \cap B \neq \emptyset$ and $P_{i,s}(A)\cap B \neq \emptyset$; together these imply that there is some $(j,t) \in A \cap B$ with $(j,t)\neq (i,s)$. Now $(j,t) \in A \cap P_{i,s}(B)$. An equivalent proof shows that if $B\neq P_{i,s}(B) \in \mcal{B}$, then $B \cap P_{i,s}(A) \neq \emptyset$.
\end{proof}

Note that this lemma also implies if $\mcal{A} \subseteq \mcal{I}_{n,k}^{r} $ is an intersecting family, then $\pi_{i,s}(\mcal{A})$ is also an intersecting family.

A family $\mcal{F} \subseteq \mcal{I}_{n,k}^{r}$ is said to be \textsl{stable} if $\pi_{i,s}(\mcal{F})=\mcal{F}$, for all $i\in[n]$ and $s \in [2,n]$. We observe that $\pi_{i,s}\circ \pi_{i,s}=\pi_{i,s}$ and $\pi_{i,s}\circ \pi_{j,t}=\pi_{j,t}\circ \pi_{i,s}$, for all $i,j \in [n]$ and $s,t \in [2,n]$. It now follows that given $\mcal{F} \subseteq \mcal{I}_{n,k}^{r}$, one can apply a sequence of compression operations to produce a stable family of the same size, $|\mcal{F}|$. A pair of families $(\mathcal{A}, \mathcal{B})$ are said to be stable if both $\mcal{A}$ and $\mcal{B}$ are stable.

The following lemma shows where intersections can be found between sets in a stable family or stable pair. This allows us to transfer our investigation into the realm of subsets of size at most $r$ in $[n]$.

\begin{lem}\label{lem:stableint}
Let $(\mcal{A}, \mcal{B})$ be a stable cross-intersecting pair of families in $\mcal{I}_{n, k}^r$.  For any $A \in \mathcal{A}$ and $B \in \mathcal{B}$, there is some $i \in [n]$ with $(i, 1) \in A \cap B$.
\end{lem}

\begin{proof}
Suppose, in hopes of a contradiction, that there are sets $A \in \mathcal{A}$, $B \in \mcal{B}$ with $A \cap B \cap ([n] \times \{1\}) = \emptyset$.  Assume that among all such pairs of sets, $|A \cap B|$ is minimum.

Since $A\cap B \neq \emptyset$, select $(i, s) \in A \cap B$.  By assumption, $s> 1$.  Since $\mathcal{A}$ is stable, $\pi_{i, s}(\mathcal{A}) = \mathcal{A}$ and so, in particular, $P_{i, s}(A) \in \mathcal{A}$.  However, $P_{i, s}(A) \cap B = A \cap B \setminus \{(i, s)\}$, contradicting the choice of the sets $A, B$.  

Therefore, for any $A \in \mathcal{A}$ and $B \in \mathcal{B}$, $A \cap B \cap ([n] \times \{1\}) \neq \emptyset$.
\end{proof}

We define the projection operation $\phi :\mcal{I}_{n,k}^{r}  \to \binom{[n]}{\leq r}$ by
\[
\phi(X) = \{i\ :\ (i,1) \in X\}.
\]
Similarly, given a family $\mcal{F}$, define $\phi(\mcal{F})=\{\phi(X)\ :\ X \in \mcal{F}\}$. As a consequence of Lemma~\ref{lem:stableint}, we have the following result.

\begin{cor}\label{cor:stableint}
For any pair of cross-intersecting families $(\mcal{A}, \mcal{B})$ in $\mcal{I}_{n, k}^r$, $(\phi(\mcal{A}), \phi(\mcal{B}))$ is a cross-intersecting pair of families in $\binom{[n]}{\leq r}$.
\end{cor}

Eventually, we will move to a larger class of families: intersecting families (or cross-intersecting families) in $\binom{[n]}{\leq n-1}$.  Because of this, we need to carefully define the property of being maximal within the family $\binom{[n]}{\leq r}$.

\begin{definition}
An intersecting family $\mcal{X} \subseteq \binom{[n]}{\leq r}$ is said to be \textsl{$r$-maximal} if for any intersecting family $\mcal{Y}$ with $\mcal{X} \subseteq \mcal{Y} \subseteq \binom{[n]}{\leq r}$, then $\mcal{X} = \mcal{Y}$.

Similarly, a cross-intersecting pair $(\mcal{S},\ \mcal{T})$ in $\binom{[n]}{\leq r}$ is  \textsl{$r$-maximal} if for any cross-intersecting pair $(\mcal{V}, \mcal{W})$ in $\binom{[n]}{\leq r}$ with $\mcal{S} \subseteq \mcal{V}$ and $\mcal{T} \subseteq \mcal{W}$, then $(\mcal{S},\ \mcal{T}) = (\mcal{V}, \mcal{W})$.
\end{definition}

Note that an intersecting family $\mcal{X} \subseteq \binom{[n]}{\leq r}$ is $r$-maximal if and only if  the pair $(\mcal{X}, \mcal{X})$ is an $r$-maximal cross-intersecting pair. Indeed, if $(\mcal{X}, \mcal{X})$ is not $r$-maximal, then there is a set $X' \in \binom{[n]}{\leq r} \setminus \mcal{X}$ so that $(\mcal{X} \cup \{X'\}, \mcal{X})$ is also cross-intersecting.  Then, $\mcal{X} \cup \{X'\}$ is an intersecting family in $\binom{[n]}{\leq r}$ and so $\mcal{X}$ is also not $r$-maximal.

A family $\mcal{X} \subseteq \binom{[n]}{\leq r} $ is called an \textsl{up-family} if it is closed under taking supersets of size no more than $r$. So $\mcal{X}$ is an up-family if for any $S \in \binom{[n]}{\leq r}$ for which there is an $X \in \mcal{X}$ with $X \subset S$, then $S$ is also in $\mcal{X}$. Any $r$-maximal intersecting family is an up-family, as are either of the families in any $r$-maximal cross-intersecting pair.

Given a family $\mathcal{B} \subseteq \binom{[n]}{ \leq r}$ (or $\binom{[n]}{\leq n-1}$), for any $\ell \geq 0$, define the \textsl{$\ell$-uniform part} of $\mathcal{B}$ to be the family $\mathcal{B}(\ell) = \{B \in \mcal{B}:\ |B| = \ell\}$.  

We denote the inverse images of a family $\mathcal{X} \subseteq \binom{[n]}{\leq r}$ under the projection operation by
$$
\phi^{-1}(\mcal{X}) = \{ A \in \mcal{I}_{n,k}^r :  \phi(A) \in \mcal{X} \}.
$$
We can compute cardinalities of $\phi^{-1}(\mcal{X})$.  For any $n, k, r, \ell$, set $C_{r, \ell}^{n, k} = \binom{n-\ell}{r-\ell}(k-1)^{r-\ell}$.  For any $n \geq r$, $k \geq 1$ and for any family $\mcal{X} \subseteq \binom{[n]}{\leq r}$,
\begin{equation}\label{eq:card-inverse-proj}
|\phi^{-1}(\mcal{X})| = \sum_{\ell = 0}^{r} \binom{n-\ell}{r-\ell}(k-1)^{r-\ell} \cdot |\mcal{X}(\ell)| = \sum_{\ell = 0}^r C_{r, \ell}^{n, k} \cdot |\mcal{X}(\ell)|.
\end{equation}

The next result can be used to find an upper bound on stable intersecting and cross-intersecting families in $\mcal{I}^{r}_{n,k}$ in terms of $r$-maximal intersecting families in $\binom{[n]}{\leq r}$.

\begin{lem}\label{lem:invmaxstab}
Let $(\mcal{S},\ \mcal{T})$ be an $r$-maximal cross-intersecting pair in $\binom{[n]}{\leq r}$, then $(\phi^{-1}(\mcal{S}),\ \phi^{-1}(\mcal{T}))$ is a stable cross-intersecting family in $\mcal{I}_{n,k}^{r}$. 
\end{lem}

\begin{proof}
Set $\mcal{A}:=\phi^{-1}(\mcal{S})$ and $\mcal{B}:=\phi^{-1}(\mcal{T})$. Given $X \in \mcal{A}$ and $Y \in \mcal{B}$, we have $\phi(X) \in \mcal{S}$ and $\phi(Y) \in \mcal{T}$. 
Since $(\mcal{S},\ \mcal{T})$ is cross-intersecting, $\phi(X) \cap \phi(Y) \neq \emptyset$. Therefore, there is some $i \in [n]$ such that $(i,1) \in X \cap Y$, and thus $\left(\mcal{A},\ \mcal{B} \right)$ is cross-intersecting. 

We now prove that $\left(\mcal{A},\ \mcal{B} \right)$ is stable.  Let $i \in [n]$, $s \in [2,n]$, and $Z \in \mcal A$ be arbitrary, so $\phi(Z) \in \mcal S$. If $(i,s) \notin Z$, then $P_{i,s}(Z)=Z \in \mcal A$.
Now suppose that  $(i,s) \in Z$. Observe that this forces $|\phi(Z)| \le r-1$. We have $P_{i,s}(Z)=Z\setminus \{(i,s)\}\cup \{(i,1)\}$. Recall that since $(\mcal S, \mcal T)$ is an $r$-maximal cross-intersecting pair, $\mcal S$ is an up-family, so $\phi(P_{i,s}(Z))= \phi(Z) \cup \{(i,1)\} \in \mcal S$ (it has cardinality at most $r$). Therefore, $P_{i,s}(Z) \in \mcal{A}=\phi^{-1}(\mcal{S})$. This shows that $\pi_{i,s}(\mcal{A})=\mcal{A}$, and thus that $\mcal{A}$ is stable. Similarly, $\mcal{B}$ is also stable. 
\end{proof}

For any family $\mcal{A}$ in $\mcal{I}_{n, k}^r$, $\mcal{A} \subseteq \phi^{-1}(\phi(\mcal{A}))$, this immediately gives the following.

\begin{cor}\label{cor:phi-inverse-phi}
Let $(\mcal{A}, \mcal{B})$ be a stable cross-intersecting pair in $\mcal{I}_{n, k}^r$.  Then, there is an $r$-maximal cross-intersecting pair $(\mcal{S}, \mcal{T})$ in $\binom{[n]}{\leq r}$  with $\phi(\mcal{A}) \subseteq \mcal{S}$ and $\phi(\mcal{B}) \subseteq \mcal{T}$.  The pair $(\phi^{-1}(\mcal{S}), \phi^{-1}(\mcal{T}))$ is a stable cross-intersecting pair in $\mcal{I}_{n, k}^r$ with $\mcal{A} \subseteq \phi^{-1}(\mcal{S})$ and $\mcal{B} \subseteq \phi^{-1}(\mcal{T})$.  If $(\mcal{A}, \mcal{B})$ is maximal, then $\mcal{A} = \phi^{-1}(\phi(\mcal{A}))$ and $\mcal{B} = \phi^{-1}(\phi(\mcal{B}))$.
\end{cor}

Lemma~\ref{lem:invmaxstab} can be used to deduce the stability and maximality of two particular families.  Recall the definition of $\mcal{H}_{n, k}^r$ from Equation~\eqref{eq:HM-family-in-clique-union} and the cross-intersecting pair $(\mfk{H}_{n, k}^r, \mfk{M}_{n, k}^r)$ defined in Equation~(\ref{eq:HM-cross-intersecting}).  We also define three special examples of families in $\binom{[n]}{\leq r}$ that correspond to the families $\mcal{H}_{n, k}^r$, and $(\mfk{H}_{n, k}^r, \mfk{M}_{n, k}^r)$ in $\mcal{I}_{n,k}^{r}$.

\begin{definition}\label{def:Vnr-cross-intersecting}
For any $r \leq n-1$, define the family
\begin{equation}\label{eq:calVr}
\mathcal{V}_{n, r} = \{[2, r+1]\} \cup \left\{X \in \binom{[n]}{\leq r} :\ 1 \in X \text{ and } X \cap [2, r+1] \neq \emptyset \right\}.
\end{equation}
For any $r \leq n$, define $\mcal{X}_{n, r} = \{[r]\}$ and $\mcal{Y}_{n, r} = \{Y \in \binom{[n]}{\leq r}:\ Y \cap [r] \neq \emptyset\}$.
\end{definition}

Note that $\mathcal{V}_{n, r}$ is an $r$-maximal intersecting family in $\binom{[n]}{\leq r}$ with $\cap \mcal{V}_{n, r} = \emptyset$ and $\phi^{-1}(\mcal{V}_{n, r}) = \mcal{H}_{n, k}^r$.  Similarly, $(\mcal{X}_{n, r}, \mcal{Y}_{n, r})$ is an $r$-maximal cross-intersecting family in $\binom{[n]}{\leq r}$ with $(\phi^{-1}(\mcal{X}_{n, r}), \phi^{-1}(\mcal{Y}_{n, r})) = (\mfk{H}_{n, k}^r, \mfk{M}_{n, k}^r)$.

\begin{cor}\label{cor:extrexarestable}
Let $n,r$ be positive integers with $2\leq r\leq n$,
\begin{enumerate}[(a)]
\item then $\left(\mfk{H}_{n,k}^{r}, \mfk{M}_{n,k}^{r} \right)$ is a maximal and stable cross-intersecting pair in $\mcal{I}_{n,k}^{r}$; and
\item if $r \leq n-1$, then $\mathcal{H}_{n,k}^r$ is a maximal and stable intersecting family in $\mcal{I}_{n,k}^{r}$.
\end{enumerate}
\end{cor}
\begin{proof}
Recall the definition of the set $\mfk{H}_{n,k}^{r} = \{ [r] \times \{1\} \} = \phi^{-1}( \{[r]\} )  = \phi^{-1}(\mcal{X}_{n, r})$ and 
$\mfk{M}_{n,k}^{r}=\phi^{-1}(\{Y \in \binom{[n]}{\leq r}\ :\ Y \cap [r] \neq \emptyset\}) = \phi^{-1}(\mcal{Y}_{n, r})$.  
Clearly these families are maximal and cross-intersecting. As noted above, the pair $(\mcal{X}_{n, r}, \mcal{Y}_{n, r})$ is $r$-maximal.  Thus, stability of $\mfk{H}_{n,k}^{r},\ \mfk{M}_{n,k}^{r}$ follows from Lemma~\ref{lem:invmaxstab}. 

The proof of part (b) is similar and follows from the fact that
\begin{align*}
\mcal{H}_{n, k}^r 
  &= \phi^{-1} \left(  \left\{ [2, r+1] \right\} \cup \left\{X \in \binom{[n]}{\leq r}:\ 1 \in X \text{ and } X \cap [2, r+1] \neq \emptyset \right\} \right) \\
 &= \phi^{-1}(\mcal{V}_{n, r}). \qedhere
\end{align*}
\end{proof}


\section{Intersection and cross-intersection in bounded sets}\label{sec:crsets}

In the proofs of Theorems~\ref{thm:main} and \ref{thm:aux}, the idea is to look at the projection of an intersecting maximal and stable family $\mathcal{F}$, and apply the Hilton-Milner Theorem (Theorem~\ref{thm:HMforSets}) and its cross-intersecting version (Theorem~\ref{thm:hmci}) to the $\ell$-uniform parts of $\phi(\mathcal{F})$;  these are then compared to either the $\ell$-uniform part of the family $\mcal{V}_r$, or the $\ell$-uniform part of the cross-intersecting family $(\mcal{X}_{n, r}, \mcal{Y}_{n, r})$.  One obstacle is that these results will not apply in the case $\ell > n/2$, which arises when $r > n/2$.  Indeed, if $r > n/2$, it is possible to have an $r$-maximal intersecting family $\mathcal{B} \subseteq \binom{[n]}{\leq r}$ with $\cap \mcal{B} = \emptyset$ and some $\ell$ with $|\mcal{B}(\ell)| > |\mcal{V}_r(\ell)|$.  To address this, we shall use a technique that was employed in~\cite{DF22years} and also in~\cite{liao2024hiltonmilnertheoremkmultisets} to simultaneously bound the number of sets arising from the $\ell$-uniform and $(n-\ell)$-uniform parts of projections.

\begin{lem}\label{lem:l-n-l}
Let $n/2 \leq r\leq n$ and let $(\mcal{S},\ \mcal{T})$ be an $r$-maximal cross-intersecting pair in $\binom{[n]}{\leq r}$. For any $\ell$ with $\ell, n-\ell \leq r$, we have 
\[
|\mcal{S}(n-\ell)|+ |\mcal{T}(n-\ell)| + |\mcal{S}(\ell)|+ |\mcal{T}(\ell)| = 2 \binom{n}{\ell}. 
\] 
\end{lem}
\begin{proof}
Let $\ell \leq r$, with $n-\ell \leq r$ as well, and consider a set $X \in \binom{[n]}{\ell} \setminus (\mcal{S}(\ell) \cup \mcal{T}(\ell) ) $. By $r$-maximality, there are sets $S \in \mcal{S}$ and $T \in \mcal{T}$ with $X \cap S = X \cap T = \emptyset$ and hence $S \subseteq X^c$ and $T \subseteq X^c$.  Since both $\mcal{S}$ and $\mcal{T}$ are up-families, and $|X^c| = n-\ell \leq r$, it follows that $X^c \in \mcal{S}(n-\ell) \cap \mcal{T}(n-\ell)$.

Conversely, for every $Y \in \mcal{S}(n-\ell) \cap \mcal{T}(n-\ell)$, since $Y \cap Y^c = \emptyset$, it follows $Y^c \notin \mcal{S}(\ell)$ and $Y^c \notin \mcal{T}(\ell)$.  Thus, $Y^c \in \binom{[n]}{\ell} \setminus (\mcal{S}(\ell) \cup \mcal{T}(\ell))$.

Thus, we have,
\begin{equation}\label{eq:setcomplement}
X \in \binom{[n]}{\ell} \setminus (\mcal{S}(\ell) \cup \mcal{T}(\ell) ) \text{ if and only if  } X^{c} \in \mcal{S}(n-\ell) \cap \mcal{T}(n-\ell).
\end{equation}

Using  the inclusion-exclusion principle and \eqref{eq:setcomplement}, it follows that 
\begin{align}\label{eq:incl-excl-l-n-l}
|\mcal{S}(n-\ell)| + |\mcal{T}(n-\ell)| 
&= |\mcal{S}(n-\ell) \cup \mcal{T}(n-\ell)| + |\mcal{S}(n-\ell) \cap \mcal{T}(n-\ell)|\notag\\ 
&=|\mcal{S}(n-\ell) \cup \mcal{T}(n-\ell)| + \binom{n}{\ell} - |\mcal{S}(\ell) \cup \mcal{T}(\ell)|.
\end{align}
Replacing $\ell$ by $n-\ell$ in Equation~\eqref{eq:incl-excl-l-n-l} gives
\begin{align*}
|\mcal{S}(\ell)|+ |\mcal{T}(\ell)|  &=  |\mcal{S}(\ell) \cup \mcal{T}(\ell)|+\binom{n}{\ell}- |\mcal{S}(n-\ell) \cup \mcal{T}(n-\ell)|. 
\end{align*} 
and adding this to Equation~\eqref{eq:incl-excl-l-n-l} gives the result.
\end{proof}

Applying Lemma~\ref{lem:l-n-l} simultaneously to an $r$-maximal cross-intersecting family in $\binom{[n]}{\leq r}$ and the family $(\mcal{X}_{n, r}, \mcal{Y}_{n, r})$ (see Definition~\ref{def:Vnr-cross-intersecting}) gives the following.

\begin{lem}\label{lem:compare-maxl}
Let $n/2 \leq r \leq n$ and let $(\mcal{S},\ \mcal{T})$ be an $r$-maximal cross-intersecting pair in $\binom{[n]}{\leq r}$.  For any $\ell$ with $\ell, n-\ell \leq r$, 
\begin{multline*}
|\mcal{S}(n-\ell)|+ |\mcal{T}(n-\ell)| + |\mcal{S}(\ell)|+ |\mcal{T}(\ell)|\\
 = |\mcal{X}_{n, r}(n-\ell)| + |\mcal{Y}_{n, r}(n-\ell)| + |\mcal{X}_{n, r}(\ell)| + |\mcal{Y}_{n, r}(\ell)|.
\end{multline*}
\end{lem}

Similarly, applying Lemma~\ref{lem:l-n-l} simultaneously to an $r$-maximal intersecting family in $\binom{[n]}{\leq r}$ and the family $\mcal{V}_{n, r}$ (and dividing by two) gives a parallel result for intersecting sets.

\begin{lem}\label{lem:compare-maxInterSets}
Let $n/2 \leq r \leq n$ and let $\mcal{B} \subseteq \binom{[n]}{\leq r}$ be an $r$-maximal intersecting family.  For any $\ell$ with $\ell, n-\ell \leq r$,
\[
|\mcal{B}(n-\ell)| + |\mcal{B}(\ell)| = |\mcal{V}_{n, r}(n-\ell)| + |\mcal{V}_{n, r}(\ell)|.
\]
\end{lem}

Before proceeding further, we include the following elementary result.
\begin{lem}\label{lem:binomdiff}
Given $n,m,r$, with $m < r\leq n$ and $2m\leq n$, we have 
\[\binom{n-m}{m}- \binom{n-r}{m} \geq 1, \] with equality if and only if $r=m+1$ and $n=2m$.
\end{lem}
\begin{proof}
We have 
\begin{align*}
\binom{n-m}{m}- \binom{n-r}{m} &\geq \binom{n-m}{m}- \binom{n-m-1}{m} \\
&= \binom{n-m-1}{m-1} \\
& \geq \binom{m-1}{m-1} &&\text{(since $2m\leq n$)}\\
&=1.
\end{align*}
Since $\binom{k}{\ell}$ is strictly increasing in $k$ for $k \geq \ell$, it is clear that equality holds if and only if  $r = m+1$ and $n = 2m$.
\end{proof}

The following lemma gives bounds on the sum of sizes of the $\ell$-uniform parts of a cross-intersecting family, using Theorem~\ref{thm:hmci}.

\begin{lem}\label{lem:cintsetrefinement}
Let $n \geq 3$, $2 \leq r \leq n$ and let $(\mcal{S}, \mcal{T})$ be an $r$-maximal, cross-intersecting pair in $\binom{[n]}{\leq r}$. For any $\ell \leq \mrm{min}\left\{r,\ \lfloor n/2 \rfloor \right\}$
\[
|\mcal{S}(\ell)| + |\mcal{T}(\ell)| \leq 
   \begin{cases} 
        \binom{n}{\ell}- \binom{n-r}{\ell} & \text{if $\ell < r$,} \\ 
         1+ \binom{n}{r}- \binom{n-r}{r} & \text{if $\ell = r$.} 
    \end{cases}
\] 
Moreover 
\begin{enumerate}
\item  if $r\geq 3$, equality holds for all $\ell$ if and only if  $(\mcal{S}, \mcal{T})$ is isomorphic to $(\mcal{X}_{n,r},\mcal{Y}_{n,r})$;
\item and if $r=2$, equality holds for all $\ell$ if and only if  $(\mcal{S}, \mcal{T})$ is isomorphic to either $(\mcal X_{n,2},\mcal Y_{n,2})$, or if $\mathcal S=\mcal T$ and both are isomorphic to \[\left\{X\in \binom{[n]}{\leq 2}\ :\ 1 \in X\right\}.\]  
  \end{enumerate}
\end{lem}

\begin{proof}
As $\mcal{S} \neq \emptyset$ and $\mcal{T} \neq \emptyset$, let $\ell_{0}, \ell_1 \leq r$ be such that $\mcal{S}(\ell_{0}) \neq \emptyset$ and $\mcal{T}(\ell_{1})\neq \emptyset$. Fix $X_0 \in \mcal{S}(\ell_{0})$ and $X_1 \in \mcal{T}(\ell_1)$. Then for any $\ell \leq r$, since $(\mcal{S}, \mcal{T})$ is cross-intersecting,
\begin{align*}
|\mcal{T}(\ell)| &\leq 
\left |  \left\{Y\in  \binom{[n]}{\ell}\ :\ Y\cap X_0 \neq \emptyset \right\} \right | = \binom{n}{\ell} -\binom{n-\ell_{0}}{\ell} \leq \binom{n}{\ell} -\binom{n-r}{\ell}, \\
|\mcal{S}(\ell)| &\leq
 \left | \left\{Y \in \binom{[n]}{\ell}:\ Y \cap X_1 \neq \emptyset \right \} \right | = \binom{n}{\ell} - \binom{n-\ell_1}{\ell} \leq \binom{n}{\ell} - \binom{n-r}{\ell}.
\end{align*}

Thus, if there exists $\ell \leq r$ such that $\mcal{S}(\ell)=\emptyset$, then
\begin{align}\label{eq:slzero}
|\mcal{S}(\ell)| + |\mcal{T}(\ell)| & \leq  \binom{n}{\ell}- \binom{n-\ell_{0}}{\ell} \leq  \binom{n}{\ell}- \binom{n-r}{\ell} .
\end{align}
Similarly, given $\ell$ such that $\mcal{T}(\ell)=\emptyset$, it follows that 
\begin{align}\label{eq:tlzero}
|\mcal{S}(\ell)| + |\mcal{T}(\ell)| & \leq  \binom{n}{\ell}- \binom{n-\ell_{1}}{\ell}\leq \binom{n}{\ell}- \binom{n-r}{\ell}.
\end{align}

We now move on to consider an $\ell$ such that both $\mcal{S}(\ell)\neq \emptyset$ and $\mcal{T}(\ell)\neq \emptyset$. For $\ell \leq r$ with $2\ell \leq n$, by Theorem~\ref{thm:hmci},
\begin{equation}\label{eq:stnonzero}
|\mcal{S}(\ell)| + |\mcal{T}(\ell)| \leq  1+\binom{n}{\ell}- \binom{n-\ell}{\ell}.
\end{equation}
Moreover, if $\ell < r$, then by Lemma~\ref{lem:binomdiff},
\begin{equation}\label{eq:stnonzero-r}
|\mcal{S}(\ell)| + |\mcal{T}(\ell)| \leq  1+\binom{n}{\ell}- \binom{n-\ell}{\ell}
	\leq \binom{n}{\ell}- \binom{n-r}{\ell}.
\end{equation}

The upper bound then follows from
\eqref{eq:slzero}, \eqref{eq:tlzero}, \eqref{eq:stnonzero}, and \eqref{eq:stnonzero-r}.

We now characterize the cross-intersecting pairs satisfying the bound with equality for all $\ell \leq \min\{r, \lfloor n/2 \rfloor \}$.
Considering $\ell = 1$ in the equality (recall that $r \ge 2$ so $\ell=r$ is not possible in this case) gives
  \[|\mcal{S}(1)|+|\mcal{T}(1)|= n- (n-r)=r.\] 
  
First, assume $r\geq 3$. Since $(\mcal{S}(1),\ \mcal{T}(1))$ is cross-intersecting with $|\mcal{S}(1)| + |\mcal{T}(1)| = r \geq 3$, we cannot have both $\mcal{S}(1) \neq \emptyset$ and $\mcal{T}(1) \neq \emptyset$ since the only cross-intersecting families of singleton sets have $\mcal{S}(1) = \mcal{T}(1) = \{x\}$, for some $x$.  Therefore, without loss of generality, we have $\mcal{T}(1)= \binom{[r]}{1}$ and $\mcal{S}(1)=\emptyset$. Since every set in $\mcal{S}(\ell)$ intersects each of the singletons in $\binom{[r]}{1}$,  then 
\[
\mcal{S}(\ell) =\begin{cases} \emptyset & \text{if $\ell < r$,} \\
\{[r]\} & \text{if $\ell=r$.}  \end{cases}
\] 
Thus, $\mcal{S}=\{[r]\}=\mcal X_{n,r}$. Since $(\mcal{S},\ \mcal{T})$ is cross-intersecting and $r$-maximal, necessarily \[\mcal{T}=\left\{X \in \binom{[n]}{\leq r} \ :\ X\cap [r] \neq \emptyset\right\}=\mcal Y_{n,r}.\]

In the case $r=2$, since $|\mcal{S}(1)|+|\mcal{T}(1)| = 2$, without loss of generality, we may assume that either
\begin{enumerate}[(i)]
 \item $\mcal{S}(1)=\emptyset$ and $\mcal{T}(1)=\{\{1\}, \{2\}\}$; or 
 \item $\mcal{S}(1)=\mcal{T}(1)=\{\{1\}\}$. 
\end{enumerate}
 In case (i), it follows that $\mcal{S}=\{[2]\}=\mcal X_{n,2}$ and $\mcal{T}=\{X \in \binom{[n]}{\leq 2} \ :\ X\cap [2] \neq \emptyset \}=\mcal Y_{n,2}$. In case (ii), we must have 
\[\mcal{S} = \left\{X \in \binom{[n]}{\leq 2} \ :\ 1\in X \neq \emptyset \right\}= \mcal{T}.\qedhere\] 
\end{proof}

We also need the following technical result to prove our main results.
\begin{lem}\label{lem:ineq}
For any integers $n$, $r \leq n$ and $1 \leq \ell < n/2 $, $C_{r,\ell}^{n,k} > C_{r, n-\ell}^{n,k}$.
\end{lem}
\begin{proof}
This follows from expanding the formula 
\begin{align*}
C_{r, n-\ell}^{n,k} & = \binom{n-(n -\ell) }{r-(n-\ell)} (k-1)^{r-(n-\ell)} \\
&= \binom{\ell }{r-n+\ell} (k-1)^{r-n+\ell} \\
&= \binom{\ell }{\ell- (r-n+\ell)} (k-1)^{r-n+\ell} \\
&= \binom{\ell }{n-r} (k-1)^{r-n+\ell} \\
&< \binom{n-\ell }{n-r} (k-1)^{r-\ell} \quad \textrm{since $\ell < n/2$}  \\
&= \binom{n-\ell}{(n-\ell)-(n-r)}(k-1)^{r-\ell}\\
&= \binom{n-\ell}{r-\ell} (k-1)^{r-\ell} \\
&=C_{r,\ell}^{n,k}. 
\end{align*}
\end{proof}

Note that Lemma~\ref{lem:ineq} holds even when $\ell < n- r$ in which case $C_{r, n- \ell}^{n, k} = 0$.

The main application of Lemma~\ref{lem:ineq} is the following optimization.

\begin{lem}\label{lem:pairwise-opt}
Let $n, \ell, r$ be integers with $r \leq n$, $1 \leq \ell < n/2$ and let $x_0, y_0 \in \mathbb{R}^+$.  For any $x, y \in \mathbb{R}$ with $x \leq x_0$ and $x+y = x_0 + y_0$,
\[
x\ C_{r, \ell}^{n, k} + y \ C_{r, n-\ell}^{n, k} \leq x_0 \ C_{r, \ell}^{n, k} + y_0 \ C_{r, n-\ell}^{n, k}
\]
with equality if and only if  $x = x_0$ and $y = y_0$.
\end{lem}

\begin{proof}
For any $x \leq x_0$ and $y$ with $x+ y = x_0 + y_0$,
\begin{align*}
x \ C_{r, \ell}^{n, k} + y \ C_{r, n-\ell}^{n, k}
	&=x_0 \ C_{r, \ell}^{n, k} + y_0  \ C_{r, n-\ell}^{n, k} + (x-x_0) \ C_{r, \ell}^{n, k} + (y-y_0)  \ C_{r, n-\ell}^{n, k}\\
	&=x_0 \ C_{r, \ell}^{n, k} + y_0  \ C_{r, n-\ell}^{n, k} + (x-x_0) \ (C_{r, \ell}^{n, k} -C_{r, n-\ell}^{n, k})\\
	&\leq x_0 \ C_{r, \ell}^{n, k} + y_0  \ C_{r, n-\ell}^{n, k} 
\end{align*}
with the last line following from the fact that $x-x_0 \leq 0$ and $C_{r, \ell}^{n, k} -C_{r, n-\ell}^{n, k} > 0$.  Furthermore, equality holds exactly when $x = x_0$, and hence $y = y_0$.
\end{proof}

\section{Hilton-Milner analogue in bounded sets}
\label{sect:ApplyHM}

The main result of this section (Lemma~\ref{lem:ApplyHM}) is essentially due to ~\cite{liao2024hiltonmilnertheoremkmultisets}. It is a Hilton--Milner analogue in the universe $\binom{[n]}{\leq r}$ that has the form we need.
All of the proof ideas in this section are essentially those from~\cite{liao2024hiltonmilnertheoremkmultisets}. We include these results here since it is not immediately obvious that the results stated in the current version of~\cite{liao2024hiltonmilnertheoremkmultisets} apply in our context.  Furthermore, as we do not require the more general setting of multisets, some of the proofs can be simplified, and the bounds on some of the variables can be extended.

In~\cite{liao2024hiltonmilnertheoremkmultisets} the authors consider a broader collection of intersecting families: any $\mathcal{B} \subseteq \binom{[n]}{\leq n-1}$.  However, for our purposes, it will suffice to consider $r$-maximal intersecting families in $\binom{[n]}{\leq r}$. 

Recall the definition of the family $\mcal{V}_{n, r}$ (Definition~\ref{def:Vnr-cross-intersecting}).
Considering the intersecting families of independent sets in $\Gamma_{n, k}$ and their projections, then $\phi(\mathcal{H}_{n, k}^r) = \mathcal{V}_{n, r}$. Our goal is to prove Lemma~\ref{lem:ApplyHM}, which states that $|\mcal{V}_{n,r}(\ell)|$ is essentially always the largest achievable size of the $\ell$-uniform part of an $r$-maximal non-canonical intersecting family in $\binom{[n]}{\leq r}$, and is achieved only by families isomorphic to $\mcal V_{n,r}$.

For any $\ell \geq 1$, the size of the $\ell$-uniform part of $\mathcal{V}_{n, r}$ is
\begin{equation}\label{eq:card-Vr}
|\mathcal{V}_{n, r}(\ell)| = 
\begin{cases}
	\binom{n-1}{\ell-1} - \binom{n-r-1}{\ell-1}	&\text{ if } \ell < r,\\
	\binom{n-1}{r-1} - \binom{n-r-1}{r-1} + 1	&\text{ if } \ell = r.
\end{cases}
\end{equation}
Note that this formula applies to all $\ell \geq 1$, using the convention that for $k > n$, $\binom{n}{k} = 0$.  Looking at the intersection properties of these families, for any $\ell$, the family $\mathcal{V}_{n, r}(\ell)$ is intersecting.  For $2 \leq  \ell < r$, $\cap \mathcal{V}_{n, r}(\ell) \neq \emptyset$, while $\cap \mathcal{V}_{n, r}(r) = \emptyset$.  We would like to compare the cardinality of the family $\mathcal{V}_{n, r}(\ell)$ with the bound given by the Hilton-Milner Theorem (Theorem~\ref{thm:HMforSets}) for the values of $\ell$ for which that theorem applies.

\begin{lem}\label{lem:vufmpart}
Let $2 \leq r \leq n-1$ and $2\leq \ell \leq \mrm{min}\{r, \lfloor n/2 \rfloor \}$, then 
\[ 
|\mcal{V}_{n, r}(\ell)| \geq \binom{n-1}{ \ell-1}- \binom{n-\ell-1}{\ell-1} +1. 
\] 
Moreover, if $r\geq 4$, xxequality holds if and only if either $\ell=r$ or $n=2r-2=2\ell$.
\end{lem}

\begin{proof}
 It is clear that equality holds in the case $\ell = r$ from Equation~\eqref{eq:card-Vr}.  For any $\ell < r$ with $\ell \leq n/2$, then, since $\ell-1 < r-1 \leq n-2$  and $2(\ell-1) \leq n-2$, by Equation~\eqref{eq:card-Vr} and Lemma~\ref{lem:binomdiff},
 \begin{align*}
 |\mathcal{V}_{n, r}(\ell)|
 	&=\binom{n-1}{\ell-1} - \binom{n-1-r}{\ell-1} = \binom{n-1}{\ell-1} - \binom{n-2 - (r-1)}{\ell-1}\\
	&\geq \binom{n-1}{\ell-1} - \binom{n-2-(\ell-1)}{\ell-1} + 1\\
	&=\binom{n-1}{\ell-1} - \binom{n-\ell-1}{\ell-1} + 1
 \end{align*}
 with equality holding only if $r -1  = (\ell-1) + 1$ and $n-2 = 2(\ell-1)$.  That is, exactly if $n$ is even, $\ell = n/2$, and $r = \ell +1$.
\end{proof}

We now consider properties of $r$-maximal intersecting families whose intersection is empty.

\begin{lem}\label{lem:nonzerodiff}
For any $n$, $2 \leq r \leq n-1$, if $\mcal{B}$ is an $r$-maximal intersecting family in $\binom{[n]}{\leq r}$ with $\cap \mathcal{B}=\emptyset$, then for any $a \in [n]$, there exists $X \in \mathcal{B}$ with $|X| = r$ and $a \notin X$.
\end{lem}
\begin{proof}
Since $a \notin \cap \mathcal{B}$, there is a set $Y \in \mathcal{B}$ with $a \notin Y$.  Since $r \leq n-1$, there exists $Z \in \binom{[n]}{r}$ with $Y \subseteq Z$ and $a \notin Z$.  Since $\mathcal{B}$ is an up-family, $Z \in \mathcal{B}$ also.
\end{proof}

\begin{lem}\cite[Lemma 2.7 and Lemma 2.8]{liao2024hiltonmilnertheoremkmultisets}\label{lem:ApplyHM}
Let $n \geq 3$, $2 \leq r \leq n-1$ and let $\mcal{B} \subseteq \binom{[n]}{\leq r}$ be an $r$-maximal intersecting family with $\cap \mathcal{B} =\emptyset$.  For every $2 \leq \ell \leq \min\{r, \lfloor n/2 \rfloor\}$,
\[
|\mathcal{B}(\ell)| \leq |\mathcal{V}_{n, r}(\ell)|.
\]
In the case that $r \geq 4$, if for every $2 \leq \ell \leq \min\{r, \lfloor n/2 \rfloor\}$ we have $|\mathcal{B}(\ell)| = |\mathcal{V}_{n, r}(\ell)|$, then $\mathcal{B}$ is isomorphic to $\mathcal{V}_{n, r}$.
\end{lem}
\begin{proof}
Fix any $\ell$ with $2 \leq \ell \leq \min\{r, \lfloor n/2 \rfloor\}$ and consider two cases, depending on $\cap \mathcal{B}(\ell)$.

If $\cap \mathcal{B}(\ell) = \emptyset$, then
\begin{align*}
|\mathcal{B}(\ell)| &\leq \binom{n-1}{\ell-1} - \binom{n-\ell-1}{\ell-1} + 1,	&&\text{(by Theorem~\ref{thm:HMforSets})}\\
	&\leq |\mathcal{V}_{n, r}(\ell)|,	&&\text{(by Lemma~\ref{lem:vufmpart})}.
\end{align*}
Furthermore, for $r \geq 4$, by Lemma~\ref{lem:vufmpart}, this inequality is strict, except in the cases that $\ell = r$, or $\ell=r-1=n/2$.

Suppose now that $\cap \mathcal{B}(\ell) \neq \emptyset$ and let $a \in \cap \mathcal{B}(\ell)$.  Since $\cap \mathcal{B} = \emptyset$, by Lemma~\ref{lem:nonzerodiff}, there is a set $X \in \binom{[n]}{r}$ with $a \notin X$ and $X \in \mathcal{B}$, so we must have $\ell<r$.  
Since every element of $\mathcal{B}(\ell)$ contains $a$ and $\mathcal{B}$ is intersecting, then
\begin{equation}\label{eq:ell-uniform-containment}
\mathcal{B}(\ell) \subseteq \{Y \in \binom{[n]}{\ell}:\ a \in Y \text{ and } Y \cap X \neq \emptyset\}. 
\end{equation}
Thus, since $\ell<r$ we have
\begin{align*}
|\mathcal{B}(\ell)| &\leq \binom{n-1}{\ell-1} - \binom{n-r-1}{\ell-1}\\
				&= |\mathcal{V}_{n, r}(\ell)|.	&&\text{(by Equation~\eqref{eq:card-Vr})}
\end{align*}
This completes the bounding part of the proof: for every $\ell$ with $2 \leq \ell \leq \min\{r, \lfloor n/2 \rfloor\}$, $|\mathcal{B}(\ell)| \leq |\mathcal{V}_{n, r}(\ell)|$.

Suppose now that $r \geq 4$, which implies $n\geq 5$, and for every $2 \leq \ell \leq \min\{r, \lfloor n/2 \rfloor\}$, $|\mathcal{B}(\ell)| = |\mathcal{V}_{n, r}(\ell)|$.  Since $n\ge 5$, we have $2 < \min\{r,  n/2 \}$. Now, by Equation~\ref{eq:card-Vr}, we $|\mcal{B}(2)|=|\mcal V_{n,r}(2)|=r\ge 4$. Since $\mcal{B}(2)$ is an intersecting family of sets of size $2$ with $|\mcal{B}(2)|\geq 4$, we must have
 $\cap \mathcal{B}(2) \neq \emptyset$. Fix $a \in \cap \mathcal{B}(2)$ and $X \in \mathcal{B}(r)$ with $a \notin X$.  Applying Equation~\eqref{eq:ell-uniform-containment}, we can relabel elements so that $a = 1$ and $X = [2, r+1]$.  Then
\[
\mathcal{B}(2) = \{Y \in \binom{[n]}{2} :\ 1 \in Y \text{ and } Y \cap [2, r+1] \neq \emptyset\} = \mathcal{V}_{n, r}(2).
\]

Since $\mathcal{B}$ is an up-family in $\binom{[n]}{\leq r}$, for every $\ell > 2$,
\[
\mathcal{B}(\ell) \supseteq \{Y \in \binom{[n]}{\ell} :\ 1 \in Y \text{ and } Y \cap [2, r+1] \neq \emptyset\} = \mathcal{V}_{n, r}(\ell) \setminus \{[2, r+1]\}.
\]
Observe that this only uses properties of up-families, so is true even when $\ell>\lfloor n/2 \rfloor$, as long as $\ell \le r$.
Thus, since $[2, r+1] \in \mathcal{B}(r)$, for any $\ell \leq r$, $\mathcal{B}(\ell) = \mathcal{V}_{n, r}(\ell)$, and therefore $\mathcal B =\mcal V_{n,r}$. 
\end{proof}


\section{Cross-intersecting pairs in $\mcal{I}_{n,k}^{r}$}
\label{sec:crintsets}

In this section, we  prove Theorem~\ref{thm:aux}. We begin with the following easy optimization result.
\begin{lem}\label{lem:summing}
Let $r$ be a positive integer and let \[(x_{1},\ldots, x_{r}), (y_{1},\ldots, y_{r}), (c_1, \ldots, c_r)  \in \bb{R}_{\geq 0}^{r}.\]
Suppose $[r]$ is partitioned into $M_1$, $M_2$, and $M_3$, such that:
\begin{itemize}
\item $\forall \ell \in M_1$, $x_{\ell} \le y_{\ell}$; and
\item there exists a bijection $\psi:M_2 \to M_3$, such that $\forall \ell \in M_2$, $$c_\ell x_{\ell}+ c_{\psi(\ell)}x_{\psi(\ell)} \le c_\ell y_{\ell}+ c_{\psi(\ell)} y_{\psi(\ell)}.$$
\end{itemize}
Then $$\sum_{\ell=1}^{r} c_\ell x_{\ell} \le \sum_{\ell=1}^{r} c_\ell y_{\ell}.$$
Furthermore, if for every $\ell \in M_2$, we have $c_\ell >c_{\psi(\ell)}$ and
$$x_{\ell}+x_{\psi(\ell)}=y_{\ell}+y_{\psi(\ell)},$$
then equality arises if and only if for every $1 \le \ell \le r$, $x_{\ell}=y_{\ell}$.
\end{lem}

\begin{proof}
We have 
\begin{align*}
\sum_{\ell=1}^{r} c_\ell x_{\ell} &= \sum_{\ell\in M_1} c_\ell x_{\ell}+\sum_{\ell \in M_2}(c_{\ell}x_{\ell}+c_{\psi(\ell)}x_{\psi(\ell)})\\
&\le \sum_{\ell\in M_1} c_\ell y_{\ell}+\sum_{\ell \in M_2}(c_{\ell}y_{\ell}+c_{\psi(\ell)}y_{\psi(\ell)})= \sum_{\ell=1}^r c_\ell y_{\ell}.
\end{align*}
For the ``furthermore", first observe that for equality to hold, we must have $x_{\ell}=y_{\ell}$ for every $\ell \in M_1$. Now let $\ell \in M_2$ be arbitrary.

By hypothesis, $$c_\ell (x_{\ell}-y_{\ell})\le c_{\psi(\ell)}(y_{\psi(\ell)}-x_{\psi(\ell)});$$
since $x_{\ell}-y_{\ell}=y_{\psi(\ell)}-x_{\psi(\ell)}$, this implies  
$$c_\ell (x_{\ell}-y_{\ell})\le c_{\psi(\ell)}(x_{\ell}-y_{\ell});$$
since $c_\ell >c_{\psi(\ell)}$, this can only happen if $x_{\ell}-y_{\ell}=0$. This also implies $y_{\psi(\ell)}-x_{\psi(\ell)}=0$.
\end{proof}

We now prove Theorem~\ref{thm:aux}.
\begin{proof}[Proof of Theorem~\ref{thm:aux}]
Clearly it is sufficient to the proof the result for maximal cross-intersecting pairs, and we claim it is also sufficient to prove the result for maximal stable cross-intersecting pairs. To prove this claim, we will assume the result is true for maximal stable cross-intersecting pairs, and show is also true for any maximal cross-intersecting pair. Let $(\mcal{C},\ \mcal{D})$ be a maximal cross-intersecting pair. Using Lemma~\ref{lem:pipreservescint}, we can apply a series of compressions on $(\mcal{C},\mcal{D})$ to obtain a stable cross-intersecting pair $(\mcal{A},\ \mcal{B})$. Thus by our assumption, we have 
\[
|\mcal{C}|+|\mcal{D}|=|\mcal{A}|+|\mcal{B}| \leq |\mfk{H}^{r}_{n,k}|+|\mfk{M}^{r}_{n,k}|.
\] 

We now show that the extremal characterization holds. If equality holds in the above equation, then by our assumption, without loss of generality, we may assume that either:
(a) $\mcal{A}=\mfk{H}_{n,k}^{r}$ and $\mcal{B}=\mfk{M}_{n,k}^{r}$; or 
(b) $r=2$ and $\mcal{A}=\mcal{B}=\{X :\ (1,1) \in X \}$. 
In case (a), $\mcal{D}=\{X\}$, for some $X \in \binom{[n]}{r}$, and thus $\mcal{C}\subseteq \{Y\in \binom{[n]}{\leq r} \:\ Y\cap X \neq \emptyset\}$. The result follows from the fact that $(\mcal{C},\ \mcal{D})$ is a maximal cross-intersecting pair. 
In case (b), since $\mcal{C}$ is compressed to $\{X :\ (1,1) \in X \}$, every independent set in $\mcal{C}$ contains an element of the form $(1,s)$. Further, if $\mcal{C}$ contains two independent sets $A$ and $B$ with $(1,s) \in A$ and $(1,s') \in B$ with $s \neq s'$, then, as $r=2$, the other element of $A$ and $B$ must be equal, so no series of compression operations will take both $A$ and $B$ to sets that include the element $(1,1)$. Thus $\cap \mcal{C} = (1,s)$, for some $s$; similarly $\cap \mcal{D} = (1,j)$ for some $j$. Since $(\mcal{C},\ \mcal{D})$ is a maximal cross-intersecting pair, $s=j$ and $\mcal{C}=\mcal{D} \cong \{X :\ (1,1) \in X \}$.

We now prove the result for maximal stable cross-intersecting pairs. Let $(\mcal{A},\ \mcal{B})$ be a stable cross-intersecting pair in $\mcal{I}_{n,k}^{r}$. By Corollary~\ref{cor:phi-inverse-phi}, we may assume without loss of generality that $(\mcal{A}, \mcal{B})$ is maximal. Let $(\mcal{S},\ \mcal{T}):= (\phi(\mcal{A}),\ \phi(\mcal{B}))$. Since $(\mcal{A}, \mcal{B})$ is maximal and stable, by Corollary~\ref{cor:phi-inverse-phi}, $(\mcal{S}, \mcal{T})$ is an $r$-maximal cross-intersecting pair with $\mcal{A} = \phi^{-1}(\mcal{S})$ and $\mcal{B} = \phi^{-1}(\mcal{T})$. Recall the definitions of $\mcal{X}_{n,r}$ and $\mcal{Y}_{n,r}$ (Definition~\ref{def:Vnr-cross-intersecting}).  We have $(\mfk{H}_{n, k}^r, \mfk{M}_{n, k}^r) = (\phi^{-1}(\mcal{X}_{n,r}),\ \phi^{-1}(\mcal{Y}_{n,r}))$.
 
We now establish some notation. For all $1 \leq \ell\leq r$, set $x_{\ell}:=|\mcal{S}(\ell)| + |\mcal{T}(\ell)|$, 
$$
y_{\ell}:=|\mcal{X}_{n,r}(\ell)|+|\mcal{Y}_{n,r}(\ell)|
= \begin{cases} 
\binom{n}{\ell}-\binom{n-r}{\ell} & \text{if $l<r$} \\
1+\binom{n}{r}-\binom{n-r}{r} & \text{if $l=r$},
\end{cases}
 $$
(the final equality can be found in Lemma~\ref{lem:cintsetrefinement}) and $c_{\ell}:=C^{n,k}_{r, \ell}$. Now, Equation~\eqref{eq:card-inverse-proj} implies

\begin{subequations}\label{eq:sizes}
\begin{align}
|\mcal{A}| + |\mcal{B}| = \sum_{\ell = 1}^r c_{\ell}x_{\ell}; \label{eq:ab}\\
|\mfk{H}_{n, k}^r| + |\mfk{M}_{n, k}^r| = \sum_{\ell = 1}^r c_{\ell}y_{\ell}. \label{eq:hm}
\end{align}
\end{subequations}
 
We will now use Lemma~\ref{lem:summing} along with Lemma~\ref{lem:cintsetrefinement} to prove our result. We proceed via a case analysis.
 
\paragraph{{\bf Case (a):} } We first assume $r\leq n/2$. Lemma~\ref{lem:cintsetrefinement} shows that $x_{\ell} \leq y_{\ell}$, for all $\ell \in [r]$.
Applying Lemma~\ref{lem:summing}, by setting $M_1=[r]$ and $M_2=M_3=\emptyset$, it follows that
\[\sum_{\ell=1}^r c_{\ell}x_{\ell} \le \sum_{\ell=1}^r c_{\ell}y_{\ell}.\] The bound now follows from the two parts of Equation~\eqref{eq:sizes}.
Lemma~\ref{lem:summing} states that the equality holds if and only if for every $1 \le \ell \le r$, we have $|\mcal{S}(\ell)| + |\mcal{T}(\ell)|=x_{\ell}= y_{\ell}$. Now the characterizaion follows from Lemma~\ref{lem:cintsetrefinement} and the fact that $(\mcal{A},\mcal{B})=(\phi^{-1}(\mcal{S}),\phi^{-1}(\mcal{T}))$.

\paragraph{{\bf Case (b):}} We now assume that $n>r>n/2$. 
Adopting the convention that $[0]$ is the empty set, we set 
$$
M_{1}= \begin{cases}
[n-r-1] \cup \{n/2\} & \text{if $n \equiv 0 \pmod 2$}, \\
[n-r-1] & \text{otherwise,}
\end{cases}
$$
$M_{2}:= \left[ n-r, \left\lfloor \dfrac{n-1}{2} \right\rfloor \right]$ and $M_{3}= \left [ \left \lceil \dfrac{n+1}{2} \right \rceil, r \right]$. 
We define the map $\psi:M_{2} \to M_{3}$ with $\psi(\ell) = n-\ell$, for all $\ell \in M_{2}$. Lemmas~\ref{lem:l-n-l}, \ref{lem:cintsetrefinement}, \ref{lem:ineq}, and~\ref{lem:pairwise-opt} (feeding the results of Lemmas~\ref{lem:l-n-l} and~\ref{lem:cintsetrefinement} into the last of these) show that the hypotheses of Lemma~\ref{lem:summing} are satisfied. Therefore, we have \[\sum_{\ell=1}^r c_{\ell}x_{\ell} \le \sum_{\ell=1}^r c_{\ell}y_{\ell},\] with equality if and only if $x_{\ell}= y_{\ell}$, for all $\ell \in [r]$. The result, in this case, now follows from Lemma~\ref{lem:cintsetrefinement} and the fact that $(\mcal{A},\mcal{B})=(\phi^{-1}(\mcal{S}),\phi^{-1}(\mcal{T}))$.  

\paragraph{{\bf Case (c):}} We now consider the case $r=n$. We set 
$$
M_{1}= \begin{cases}
 \{n,n/2\} & \text{if $n \equiv 0 \pmod 2$} \\
\{n\} & \text{otherwise,}
 \end{cases}$$
 $M_{2}=\left[ \left \lfloor \dfrac{n-1}{2} \right \rfloor \right]$, and $M_{3}=\left[ \left \lceil \dfrac{n+1}{2} \right \rceil , n-1 \right]$. Since there is exactly one $n$-set in $\binom{[n]}{\leq n}$, it follows that $x_{n} \leq y_{n}$.  Arguing as in the previous case, using Lemmas~\ref{lem:l-n-l}, \ref{lem:cintsetrefinement}, \ref{lem:ineq}, and \ref{lem:pairwise-opt}, we check that the hypotheses of Lemma~\ref{lem:summing} are satisfied. The result, in this case, now follows from Lemma~\ref{lem:cintsetrefinement} and the fact that $(\mcal{A},\mcal{B})=(\phi^{-1}(\mcal{S}),\phi^{-1}(\mcal{T}))$.  This concludes the proof.
\end{proof}

\section{Intersecting families in $\mcal{I}_{n,k}^{r}$}\label{sect:mainproof}

In this section, we use Theorem~\ref{thm:aux} as the main tool to prove Theorem~\ref{thm:main}.
When we use Theorem~\ref{thm:aux}, we compare the size of a cross-intersecting pair of families in $\mcal{I}_{n-1, k}^{r-1}$ with the family $\mcal{H}_{n, k}^r$.  For this, note that 
\[
|\mfk{H}_{n-1, k}^{r-1}| + |\mfk{M}_{n-1, k}^{r-1}| = 1 + |\{X \in \mcal{I}_{n-1, k}^{r-1}:\ X \cap ([r-1] \times \{1\}) \neq \emptyset\}|
\]
while by decomposing $\mcal{H}_{n, k}^r$ into the sets that contain $(1, 1)$ and the one that does not, we see
\begin{align}
|\mcal{H}_{n, k}^r| &= 1 + |\{Y \in \mcal{I}_{n, k}^r:\ (1, 1) \in Y,\ Y \cap ([2, r+1] \times \{1\}) \neq \emptyset\}| \notag\\
	&= 1 + |\{X \in \mcal{I}_{n-1, k}^{r-1}:\ X \cap ([r] \times \{1\}) \neq \emptyset\}| \notag\\
	&> |\mfk{H}_{n-1, k}^{r-1}| + |\mfk{M}_{n-1, k}^{r-1} |, \label{eq:compare-HvsM}
\end{align}
as long as $n>r$ (recall that $\mcal{H}_{n, k}^r$ is defined only when $n>r$).
\begin{proof}[Proof of Theorem~\ref{thm:main}]
Let $\mcal{B}$ be an intersecting family in $\mcal{I}_{n, k}^r$ with $\cap \mcal{B}= \emptyset$. By Lemma~\ref{lem:pipreservescint}, we can obtain a stable intersecting family $\mcal{C}$ with $|\mcal{B}| = |\mcal{C}|$, via a sequence of compression operations.  We  need to consider two cases:
(a) $\cap \mcal{C} \neq \emptyset$ (this of course will not arise if $\mcal B$ was already stable); or
(b) $\cap \mcal{C} = \emptyset$. 

\paragraph{{\bf Case (a):} } Suppose that $\cap \mcal{C} \neq \emptyset$. In this case, looking at the sequence of compression operations for the last time the family had empty intersection (and relabelling points if necessary), we let $\mcal{F}$ be an intersecting family with $\cap \mcal{F}=\emptyset$, $|\mcal{F}| = |\mcal{B}|$ be such that $\pi_{1,2}(\mcal{F})=\mcal{C}$, so $\cap \mcal{C}=\{(1,1)\}$. This happens if and only if the following are true:
\begin{enumerate}[(i)]
\item $\mcal{F}=\mcal{F}_{1} \cup \mcal{F}_{2}$, where
$\mcal{F}_{i}=\{X\in \mcal{F} :\ (1,i) \in X\}$, for $i \in \{1,2\}$; and
\item $P_{1,2}(X) \notin \mcal{F}_{1}$, for all $X \in \mcal{F}_{2}$.
\end{enumerate} 

For each $i\in \{1,2\}$, define $\tilde{\mcal{F}}_{i}=\{X\setminus \{(1,i)\} \ :\ X \in F_{i} \}$. Since $\cap \mcal{F} = \emptyset$, $\tilde{\mcal{F}}_1, \tilde{\mcal{F}}_2 \neq \emptyset$.  By the second condition above, we note that $\tilde{\mcal{F}}_{1} \cap \tilde{\mcal{F}}_{2}=\emptyset$. Since $\mcal{F}$ is intersecting, $(\tilde{\mcal{F}}_{1},\ \tilde{\mcal{F}}_{2})$ is a non-empty cross-intersecting pair in $\mcal{I}^{r-1}_{n-1,k}$.Thus,
\begin{align*}
|\mcal{B}|&=|\mcal{F}| = |\mcal{F}_1| + |\mcal{F}_2|\\
	&=|\tilde{\mcal{F}}_{1}|+|\tilde{\mcal{F}}_{2}|\\
	&\leq |\mfk{H}^{r-1}_{n-1,k}|+|\mfk{M}^{r-1}_{n-1,k}|. &&\text{(by Theorem~\ref{thm:aux})}
\end{align*}
If  $r < n$ (recall that $\mcal H_{n,k}^r$ is only defined for $r<n$), then by Equation~\eqref{eq:compare-HvsM}, 
\[
|\mfk{H}^{r-1}_{n-1,k}|+|\mfk{M}^{r-1}_{n-1,k}| <|\mcal{H}^{r}_{n,k}|,
\]
completing the proof in this case.

If $r = n$, then
\begin{align*}
|\mcal{F}_1| + |\mcal{F}_2| & = |\tilde{\mcal{F}}_1| + |\tilde{\mcal{F}}_2|\\
 & \leq |\mfk{H}^{n-1}_{n-1,k}|+|\mfk{M}^{n-1}_{n-1,k}|\\
	&=|\{[n-1] \times \{1\}\}| + |\{X \in \mcal{I}_{n-1, k}^{n-1} :\ X \cap [n-1] \times \{1\} \neq \emptyset\}|\\
	&= 1 + k^{n-1} - (k-1)^{n-1} \\
	&\leq k^{n-1} - (k-1)^{n-1} + (k-1),
\end{align*}
with equality in the final equation only if $k=2$. In the case $k=2$ and equality holding throughout, Theorem~\ref{thm:aux} implies that
$\mathcal{F}_1\cong \{[n] \times \{1\}\} $  and $\mathcal{F}_2=\{X \in \mcal{I}_{n, k}^n:\ (1, 2) \in X \text{ and } X \cap F \neq \emptyset\}$, where $F$ is the solitary element of $\mathcal{F}_{1}$.  

\paragraph{{\bf Case (b):} } Suppose that $\cap \mcal{C}=\emptyset$. Since $\mcal{C}$ is stable, by Lemma~\ref{lem:stableint}, $\phi(\mcal{C})$ is an intersecting family of sets, so we can consider an $r$-maximal intersecting set $\mcal{S}$ in $\binom{[n]}{\leq r}$ with $\phi(\mcal{C})\subseteq \mcal{S}$.

As in the proof of Theorem~\ref{thm:aux}, we apply Lemma~\ref{lem:summing}. This time, we set $x_{\ell}=|\mcal{S}_{\ell}|$,  
$$
y_{\ell}=|\mcal{V}_{n, r}(\ell)|
= \begin{cases} 
\binom{n-1}{\ell-1}- \binom{n-r-1}{\ell -1} & \text{if $\ell<r$}\\
\binom{n-1}{r-1}- \binom{n-r-1}{r -1}+1 & \text{if $\ell=r$,}
\end{cases}
$$ 
and $c_{\ell}=C_{r,\ell}^{n,k}$. Note that since $\cap \mcal{C} = \emptyset$, then $\cap \mcal{S} = \emptyset$ also.
By Lemma~\ref{lem:invmaxstab}, we may assume that $\mcal{C}= \phi^{-1}(\mcal{S})$, since $\mcal{C}$ is maximal and
$\phi^{-1}(\mcal{S})$ will have empty intersection since $\cap \mcal{S} = \emptyset$. 

Thus, by Equation~\eqref{eq:card-inverse-proj},
\begin{subequations}\label{eq:size}
\begin{align}
|\mcal{C}| &= \sum_{\ell = 0}^{r} c_{\ell}x_{\ell},\\
|\mcal{H}_{n,k}^{r}| &= \sum_{\ell = 0}^{r} c_{\ell}y_{\ell}.
\end{align}
\end{subequations}

\paragraph{{\bf Subcase 1:}} Consider first the case $r \leq n/2$. Lemma~\ref{lem:ApplyHM}
shows that $y_{\ell} \geq x_{\ell}$, for all $\ell \in [r]$. Applying Lemma~\ref{lem:summing}, by setting $M_{1}=[r]$, $M_{2}=M_{3}=\emptyset$, we conclude that
$|\mcal{H}_{n,k}^{r}|= \sum c_{\ell} y_{\ell} \geq \sum c_{\ell} x_{\ell}=|\mcal{C}|$.
If equality holds, Lemma~\ref{lem:summing} implies that $|\mcal{S}(\ell)|=x_{\ell}=y_{\ell}=|\mcal{V}_{n,r}(\ell)|$, for all $\ell\in [r]$. Now, under the constraint $r\geq 4$, using Lemma~\ref{lem:ApplyHM}, we see that $|\mcal{C}|=|\mcal{H}_{n,k}^{r}|$ implies that $\mcal{S} \cong \mcal{V}_{n,r}$. Using $\phi^{-1}(\mcal{S})=\mcal{C}$, we conclude that $|C|=|\mcal{H}_{n,k}^{r}|$ if and only if $C\cong\mcal{H}_{n,k}^{r}$.

\paragraph{{\bf Subcase 2:} } Now, we consider the case $n/2 < r < n$. Adopting the convention that $[0]$ is the empty set, we set \[
M_{1}= \begin{cases}
[n-r-1] \cup \{n/2\} & \text{if $n \equiv 0 \pmod 2$} \\
[n-r-1] & \text{otherwise,}
 \end{cases}
 \] 
 $M_{2}:= \left[ n-r,  \left \lfloor \dfrac{n-1}{2} \right \rfloor \right]$ and $M_{3}= \left[ \left\lceil \dfrac{n+1}{2} \right\rceil, r \right]$. We define the map $\psi:M_{2} \to M_{3}$ by $\psi(\ell) = n-\ell$, for all $\ell \in M_{2}$. Lemmas~\ref{lem:compare-maxInterSets}, \ref{lem:ineq}, \ref{lem:pairwise-opt}, and \ref{lem:ApplyHM} 
 show that the hypotheses of Lemma~\ref{lem:summing} are satisfied. Applying Lemma~\ref{lem:summing}, we conclude that $|\mcal{H}_{n,k}^{r}|= \sum c_{\ell} y_{\ell} \geq \sum c_{\ell} x_{\ell}=|\mcal{C}|$, with equality if and only if $|\mcal{H}_{n,k}^{r}|= \sum c_{\ell} y_{\ell} = \sum c_{\ell} x_{\ell}=|\mcal{C}|$, for all $\ell \in [r]$. Using Lemma~\ref{lem:ApplyHM} as in the previous subcase, we can conclude that the result is true in this subcase as well.

\paragraph{{\bf Subcase 3:}} Now consider the case $r = n$.  Since $\mathcal{S} \subseteq \binom{[n]}{\leq n}$ is $n$-maximal, then $[n] \in \mathcal{S}$.  Set $\mathcal{S}' = \mathcal{S} \setminus \{[n]\}$. We observe that $\mathcal{S}' \subseteq \binom{[n]}{\leq n-1}$ is an $(n-1)$-maximal intersecting set with $\cap \mathcal{S}' = \emptyset$ also. Set $x_{\ell}=|\mathcal{S}(\ell)|$ and 
$$
y_{\ell}= \begin{cases} 
|\mcal{V}_{n,n-1}(\ell)| & \text{for $\ell<n$}\\ 
1 & \text{otherwise.}
\end{cases}
$$ 
We also define 
$$
M_{1}= \begin{cases}
 \{n,n/2\} & \text{if $n \equiv 0 \pmod 2$} \\
\{n\} & \text{otherwise,}
 \end{cases}
 $$ 
 $M_{2}=\left[ \left \lfloor \dfrac{n-1}{2} \right \rfloor \right]$, and $M_{3}=\left[ \left \lceil \dfrac{n+1}{2} \right \rceil, n-1 \right]$. Since there is exactly one $n$-set in $\binom{[n]}{\leq n}$, it follows that $x_{n}=1= y_{n}$. We define the map $\phi:M_{2} \to M_{3}$ by $\phi(\ell) = n-\ell$, for all $\ell \in M_{2}$. Lemmas~\ref{lem:compare-maxInterSets}, \ref{lem:ineq}, \ref{lem:pairwise-opt} and \ref{lem:ApplyHM}
 show that the premise of Lemma~\ref{lem:summing} is satisfied. Applying Lemma~\ref{lem:summing} and \eqref{eq:card-inverse-proj}, we conclude that 
\[
 |\phi^{-1}(\{[n]\})|+|\phi^{-1}(\mcal{V}_{n,n-1})|= \sum c_{\ell} y_{\ell} \notag
  \geq \sum c_{\ell} x_{\ell}=|\phi^{-1}(\{[n]\})|+|\phi^{-1}(\mcal{S'})|=|\mcal{C}|
 \]
 with equality implying that $|\mcal{S}'(\ell)|=x_{\ell}=y_{\ell}=|\mcal{V}_{n,n-1}(\ell)|$, for all $\ell \in [n-1]$. Provided $n-1\geq 4$, by Lemma~\ref{lem:ApplyHM}, in the extremal case, we have $S' \cong \mcal{V}_{n,n-1}$, and therefore 
 \[\mcal{C}= \phi^{-1}([n]) \cup \phi^{-1}(\mcal{S}') \cong \phi^{-1}([n]) \cup \phi^{-1}(\mcal{V}_{n,n-1}).\] This concludes the proof. 
\end{proof}

\section{Application}\label{sect:app}

In~\cite{claws}, Feghali, Johnson and Thomas consider the largest of intersecting families of $r$-independent sets in depth-two claws. The EKR Theorem was established for these sets, provided that $r$ is relatively small. Using our Hilton-Milner Theorem for unions of complete graphs, we can extend this to determine all values of $r$ for which the depth-two claw is $r$-EKR.

\begin{definition}
A \textsl{depth-two claw} is a tree that has one vertex called the root and every other vertex has degree 1 or 2, with every leaf  at distance 2 from the root. We denote a depth-two claw with $n$ leaves by $G_n$. 
\end{definition}

Note that the depth-two claw $G_n$ can also be thought of as a special kind of spider graph, where the spider has $n$ legs all of length $2$, or as a star with $n$ leaves in which every edge has been subdivided.

If the root is removed from the depth-two claw $G_n$, the remaining graph will be isomorphic to  $\Gamma_{n,2} = \cup_{i=1}^n K_2$. Define $\mcal I_n^r$ to be the set of all $r$-independent sets in $G_n$. A family $\mathcal{A} \subseteq \mcal{I}_n^{r}$ is called \textsl{canonically intersecting} if every set $A \in \mathcal{A}$ contains a vertex $x$ that is a leaf in $G_n$. A canonically intersecting set has size
\[
\binom{n-1}{r-1} 2^{r-1} + \binom{n-1}{r-2}.
\]
The first summand counts the sets that contain $x$, but not the root, and the second summand counts the sets that contain both $x$ and the root.

We begin with  a technical result.
\begin{lem}\label{lem:tech}
Let $4 \le r \le n-1$. Let $m \in \{r,r-1,r-2\}$. If $m<n-1$ then 
\[
 \sum_{j=0}^{r-1} \binom{m}{j} \binom{n-m-1}{r-j-1} 2^{r-j-1}    >  1 + \binom{n-1}{r-1} .
\]
If $m=n-1$ then
\[
 \sum_{j=0}^{r-1} \binom{m}{j} \binom{n-m-1}{r-j-1} 2^{r-j-1}  = \binom{n-1}{r-1} .
\]
\end{lem}
\begin{proof}
The statement for $m=n-1$ follows since the only non-zero term in the sum occurs when $j = r-1$. 

If $m=n-2$ then there are two non-zero terms, and the sum becomes $$\binom{n-2}{r-2}\binom{1}{1}2^1+\binom{n-2}{r-1}\binom{1}{0}2^0=2\binom{n-2}{r-2}+\binom{n-2}{r-1}=\binom{n-1}{r-1}+\binom{n-2}{r-2}.$$ Since $n-2 >r-2>0$ we have $\binom{n-2}{r-2} \ge n-2 \ge 3$, completing the proof in this case.

Now we assume $m \le n-3$. Clearly, for any $j$, $\binom{m}{j} \binom{n-m-1}{r-j-1} 2^{r-j-1} \geq \binom{m}{j} \binom{n-m-1}{r-j-1}$, but we need a slightly stronger result, so we consider the term of the summation that arises when $j=r-3$ separately. This is $$\binom{m}{r-3}\binom{n-m-1}{2}2^{2}.$$ Observe that since $m \le n-3$, we have $n-m-1\ge 2$, so the second binomial coefficient in this product is at least $1$, and also since $m>r-3>0$, the first binomial coefficient is at least $m$. This easily implies that $$\binom{m}{r-3}\binom{n-m-1}{2}2^{2}>\binom{m}{r-3}\binom{n-m-1}{2}+1.$$
Using this, we can see that 

\begin{align*}
 \sum_{j=0}^{r-1} \binom{m}{j} \binom{n-m-1}{r-j-1} 2^{r-j-1}&>
1 +  \sum_{j=0}^{r-1} \binom{m}{j} \binom{n-m-1}{r-j-1} \\
&= 1 + \binom{n-1}{r-1}.
 \end{align*}
The last equality follows from the Vandermonde binomial identity.
\end{proof}

We now show that depth-two claws have the $r$-EKR property unless $r=n$, and as long as $r \neq n-1$ they in fact have the strict $r$-EKR property. This means that the largest intersecting family from $\mcal I_n^{r}$ is the canonical intersecting system except when $r=n$, and furthermore that canonical intersecting systems are the only intersecting families of the maximum cardinality except when $r=n-1$.

\begin{thm}
Let $G_n$ be a depth-two claw with $n$ leaves. Let $1 \leq r \leq n-1$, then $G_n$ is $r$-EKR.
Moreover, if $4 \leq r < n-1$, then $G_n$ is strict $r$-EKR.
\end{thm}
\begin{proof}
Let $G_n$ be a depth-two claw with $n$ leaves. It was shown in~\cite{claws} that $G_n$ is $r$-EKR, provided that $2r-1 \le n$; this immediately implies the first statement when $r \in \{1,2\}$, and also when $r=3$ except in the case $n=4$. An exhaustive search confirms the statement in this final small case. Henceforth in this proof we assume $r \ge 4$.

Denote the root vertex by $c$, the vertices adjacent to $c$ are labelled by $b_1,b_2, \dots, b_n$ and the leaves are labelled by $a_1, a_2, \dots, a_n$ where the leaf adjacent to $b_i$ is denoted by $a_i$.

Let $\mathcal{A} \subseteq \mcal{I}^{r}_n$ be intersecting. This family can be partitioned into two families
\[
\mathcal{B}=\{ A \in \mathcal{A} \, : \, c \not \in A\},
\qquad
\mathcal{C} = \{  A \in \mathcal{A} \, : \, c \in A \}.
\]
Note that for any $C \in \mathcal{C}$, since $c \in C$ for all $i \in [n]$, $b_i \not \in C$, this implies $|\mathcal{C}| \leq \binom{n}{r-1}$.

We will show that $| \mathcal{A} | \leq \binom{n-1}{r-1} 2^{r-1} + \binom{n-1}{r-2} $ and that this bound is met only if $\mathcal{A}$ is canonically intersecting.
We can assume that $\mathcal{A}$ is maximal, so every $r$-independent set in $G_n$ that intersects every set in $\mathcal{A}$ belongs to $\mathcal{A}$.

We will consider 2 cases, depending on whether $\cap\mcal B=\emptyset$. 

\paragraph{{\bf Case (a):} } Assume $\cap\mcal B \neq \emptyset$; say $x \in B$ for every $B \in \mathcal{B}$, where $x$ is a vertex of $G_n$.
In this case, if $x \in C$ for all $C \in \mathcal{C}$, then, since $\mathcal{A}$ is maximal, $\mathcal{A}$ is canonical. If there is a $C \in \mathcal{C}$ with $x \not \in C$, then, without loss of generality $C=\{c, a_1, a_2, \dots, a_{r-1}\}$. This means that each $B \in \mathcal{B}$ must intersect with $C$, so must contain at least one of the vertices $a_1,a_2, \dots, a_{r-1}$.

Let $\mathcal{T}$ be the set of all the $r$-independent sets that contain $x$ but not $c$. We have 
\[
| \mathcal{T} | = \binom{n-1}{r-1} 2^{r-1},
\]
and $\mathcal{B} \subset \mathcal{T}$. 

If $x$ is not $b_i$, for some $i \in [r-1]$ (so $x$ is not adjacent to any of the $a_1,a_2, \dots, a_{r-1}$ in $C$), then 
\[
\sum_{j=0}^{r-1} \binom{r-1}{j} \binom{n-r}{r-j-1} 2^{r-j-1} > \binom{n-1}{r-1}
\] 
elements in $\mathcal{T}$ do not intersect $C$. The last inequality follows from Lemma~\ref{lem:tech} with $m=r-1$; since $r<n$, we have $m < n-1$. 
 
Similarly, if $x$ is one of the $b_i$ with $i \in [r-1]$, then 
\[
\sum_{j=0}^{r-2} \binom{r-2}{j} \binom{n-r+1}{r-j-1} 2^{r-j-1} > \binom{n-1}{r-1}
\] 
elements in $\mathcal{T}$ do not intersect $C$ (again, the last inequality follows from Lemma~\ref{lem:tech} with $m=r-2$, since $r < n$, we have $m < n-1$).
We have in both cases that
\[
|\mathcal{B}| < \binom{n-1}{r-1} 2^{r-1} -\binom{n-1}{r-1}.
\]

Thus we have that
\begin{align*}
| \mathcal{A} | &= | \mathcal{B} | +| \mathcal{C} |  \\
& <  \binom{n-1}{r-1} 2^{r-1} - \binom{n-1}{r-1}  +  \binom{n}{r-1} \\
& = \binom{n-1}{r-1} 2^{r-1} + \binom{n-1}{r-2}.
\end{align*}
So this family is smaller than the canonical system.

\paragraph{{\bf Case (b):} } Assume $\cap\mcal B = \emptyset$, so we can apply our Hilton-Milner Theorem as $\mathcal{B}$ is equivalent to a non-canonically intersecting family of $r$-independent sets in $\Gamma_{n,2}$, which lies in $\mathcal I_{n,2}^r$. Thus, using Theorem~\ref{thm:main} to bound the size of $\mathcal{B}$, along with Lemma~\ref{lem:tech} with $m=r$, adding the additional hypothesis that $r<n-1$, and using our original formula for $|\mathcal H_{n,k}^r|$ rather than Equation~\eqref{eq:card-HM-family}, we have 
\begin{align*}
| \mathcal{A} | &= | \mathcal{B} | +| \mathcal{C} |  \\
& \leq 
\binom{n-1}{r-1} 2^{r-1} + 1 - \sum_{j=0}^{r-1} \binom{r}{j} \binom{n-r-1}{r-j-1} 2^{r-j-1} +  \binom{n}{r-1} \\
&< \binom{n-1}{r-1} 2^{r-1} - \binom{n-1}{r-1}  +  \binom{n}{r-1} \\
&=\binom{n-1}{r-1} 2^{r-1} + \binom{n-1}{r-2} ,
\end{align*}
which is the size of the canonical intersecting system. 

Finally, we need to consider the possibility within this case that $r=n-1$. We identify the induced subgraph $G_{n} \setminus \{c\}$ with $\Gamma_{n,2}$, via the isomorphism $\psi:G_{n} \setminus \{c\} \to \Gamma_{n,2}$, given by $\psi(a_{i})=(i,1)$ and $\psi(b_{i})= (i,2)$, for all $i\in [n]$. Define $\tilde{\mcal{C}}=\{C \setminus \{c\} \:\ C \in \mcal{C}\}$. We can transfer the projection and compression operations to the universe of subsets of the vertex set of $G_{n} \setminus \{c\}$. 
Since elements of any set in $\tilde{\mcal{C}}$ is made up of leaves, $\mcal{C}$ is stable under compression operations. Since our compression operation respect intersection, we may replace $\mcal{B}$ with a stable family, while leaving $\mcal{C}$ the same. We now assume without loss of generality that $\mcal{B}$ is stable. Given $i\in [n]$, define $R_{i}:=\{a_{1},\ldots, a_{n}\}\setminus \{a_{i}\}$ and define 
\[
\mcal{H}_{i}= 
\{X \in \mcal{I}^{n}_{n-1}\ :\ a_{i}\in X,\  c\notin X,\ \textrm{ and } \ R_i \cap X\neq \emptyset\}
\cup R_{i} .
\] 
Using Theorem~\ref{thm:main}, it follows that sets $\mcal{H}_{i}$ are the only size $|\mcal{H}_{n,2}^{n-1}|$ stable maximum intersecting families of $(n-1)$-independent sets in $G_{n} \setminus \{c\}$. Proceeding as in the $r<n-1$ case, using $|\mcal{C}| \leq \binom{n}{n-2}$, Theorem~\ref{thm:main}, and Lemma~\ref{lem:tech} with $r=m=n-1$ yields
\begin{align}
| \mathcal{A} | = | \mathcal{B} | +| \mathcal{C} | \leq |\mcal{H}_{n,2}^{n-1}|+|\mcal{C}|&=\binom{n-1}{n-2}2^{n-2}-\binom{n-1}{n-2}+1 +\binom{n}{n-2}\nonumber\\
&=\binom{n-1}{r-1} 2^{r-1} + \binom{n-1}{r-2} +1
\label{eq:n-1}
\end{align}
 with equality if and only if $\mcal{B} \cong \mcal{H}_{1}$ and $|\mcal{C}|= \binom{n}{n-2}$. 
Assume that $\mcal{B} \cong \mcal{H}_{1}$ and $|\mcal{C}|= \binom{n}{n-2}$.
Since $\mcal{B}$ is stable, we may assume without loss of generality that $\mcal{B}=\mcal{H}_{1}$. This means that the set $X:=\{a_{1},a_{2}\} \cup \{b_{i} \ :\ i \in [3,n-1]\}$ must be in $\mcal{B}$. If we define $C:=\{c\} \cup \{a_{i} \ :\ i \in [3,n]\}$, then, as $C\cap X= \emptyset$, we must have $C \notin \mcal{C}$ which implies $|\mcal{C}| < \binom{n}{n-2}$. Thus, the inequality in \eqref{eq:n-1} is strict. This concludes the proof.
\end{proof}

We conclude this section by showing that the condition $r \le n-1$ in the above theorem is necessary.

\begin{prop}\label{prop:r=n}
Let $r=n$ and $n>2$. The depth-2 claw $G_n$ with $n$ leaves, is not $r$-EKR.
\end{prop}
\begin{proof}
Define $\mathcal{B}$ to be all $n$-independent sets in $G_n$ that do not contain $c$ and have at least 
$\lceil \frac{n+1}{2} \rceil$ vertices from $\{a_1, a_2, \dots, a_n\}$. If $n$ is even, $\mathcal{B}$ will also contain all the $n$-independent sets in $G_n$ that contain $a_1$ and $\frac{n}{2}-1$ vertices from $\{a_2, \dots, a_n\}$. In either case, $|\mathcal{B}| = \frac{2^n}{2}$. This is the largest possible size of $\mathcal{B}$, since if $B$ is any $n$-independent set in $G_n$ that does not contain $c$, exactly one of $B$ or $B^c-\{c\}$ is in $\mathcal{B}$.  Define $\mathcal{C}$ to be all $n$-independent sets that include $c$. Then $\mathcal{B} \cup \mathcal{C}$ is intersecting and has size
\[
\frac{2^n}{2} + \binom{n}{n-1} = 2^{n-1} + n.
\]
In this case the size of a canonical set is $2^{n-1} + (n-1)$.
\end{proof}

\section{Further Work}
\label{sec:Further}

As Holroyd, Spencer, and Talbot~\cite{HST05} showed that a graph that is the union of $n \geq r$ complete graphs (each of order at least $2$) is $r$-EKR, a natural question is whether our Hilton-Milner Theorem-type result, Theorem~\ref{thm:main}, would extend to graphs that are the union of cliques of different sizes (all of order at least $2$). Here the canonical $r$-independent sets would be the set of all $r$-independent sets that contain a fixed vertex from a smallest clique. 

Another possible question is to ask whether, for any $k \geq 3$, the graph formed by taking $n$ copies of $K_k$ and attaching a ``root'' vertex to exactly one vertex from each copy of $K_k$ would be $r$-EKR for any $r \leq n-1$?  Taking $k = 2$ would give a depth-2 claw with $n$ leaves.  In these graphs, a vertex with the most $r$-independent sets containing it is always either the root or a vertex not adjacent to the root, but it is not clear which of these two canonically intersecting sets have the maximum size.   The number of $r$-independent sets containing the root is $\binom{n}{r-1}(k-1)^{r-1}$, while the number of $r$-independent sets containing a fixed vertex not adjacent to the root is $\binom{n-1}{r-2}(k-1)^{r-2} + \binom{n-1}{r-1}k^{r-1}$.  Consider the difference of these two quantities:
\begin{align*}
\binom{n-1}{r-2}&(k-1)^{r-2} + \binom{n-1}{r-1}k^{r-1}-\binom{n}{r-1}(k-1)^{r-1} \\
	&=\binom{n-1}{r-2}\frac{1}{r-1}\left((r-1)(k-1)^{r-2} + (n-r+1)k^{r-1} - n(k-1)^{r-1} \right).
\end{align*}
For example, in the case that $r = 2$, 
\begin{align*}
(r-1)(k-1)^{r-2} +& (n-r+1)k^{r-1} - n(k-1)^{r-1} \\
	&=1+(n-1)k - n(k-1)\\
	&=n+1-k
\end{align*}
which is negative when $k > n+1$ and non-negative for $k \leq n + 1$.  In general, one can show that for $k \leq \frac{n-2}{\ln(n/2)}$, then for all $r \in [2, n-1]$, a vertex not adjacent to the root will be contained in the maximum number of $r$-independent sets.  Further, for any $\varepsilon > 0$, if $\frac{(1+\varepsilon)(n-2)}{\ln(n/2)} \leq k \leq n+1$, then for some values of $r$, the root has the most $r$-independent sets containing it and for others it is a vertex not adjacent to the root that has the most $r$-independent sets.  When $k > n+1$, then for all $r \in [2, n-1]$, the root has the most $r$-independent sets.

Even in the case $k = 3$, where a vertex not adjacent to the root will have the most $r$-independent sets for any $n \geq 3$ and $2 \leq r \leq n-1$, our proof technique does not extend directly for all values of $r$.

It may be possible to apply the method adapted here from Deza and Frankl~\cite{DF22years} and Liao, Lv, Cao and Lu in~\cite{liao2024hiltonmilnertheoremkmultisets} to prove a Hilton-Milner-type theorem for other objects that have a natural projection to sets.

\section*{Acknowledgements}

The authors are all indebted to the support of the Pacific Institute for Mathematical Sciences (PIMS), through the establishment of the Collaborative Research Group on Movement and Symmetry in Graphs. In connection with this CRG, this paper has the report identifier PIMS-20251123-CRG36.

\bibliographystyle{plain}
\bibliography{bibtex.bib}                        
\end{document}